\documentclass[a4paper,10pt,leqno]{amsart}
\title{Survey on analytic and topological torsion}

\author{Wolfgang L\"uck}
        \address{Mathematisches Institut der Universit\"at Bonn\\
                Endenicher Allee 60\\
                53115 Bonn, Germany}
         \email{wolfgang.lueck@him.uni-bonn.de}
          \urladdr{http://www.him.uni-bonn.de/lueck}

         \date{November, 2015}

\usepackage{amsmath}
\usepackage{amsfonts}
\usepackage{amssymb}
\usepackage{amscd}
\usepackage{amsthm}
\usepackage{amsbsy}
\usepackage{graphicx}
\usepackage{bm}

\newtheorem{theorem}{Theorem}[section]

\newtheorem{lemma}[theorem]{Lemma}
\newtheorem{corollary}[theorem]{Corollary}

\theoremstyle{definition}
\newtheorem{definition}[theorem]{Definition}

\theoremstyle{remark}
\newtheorem{example}[theorem]{Example}
\newtheorem{remark}[theorem]{Remark} 

\usepackage{hyperref}
\usepackage{color}
\usepackage{pdfsync}
\usepackage{calc}
\usepackage{enumerate,amssymb}
\usepackage[arrow,curve,matrix,tips,2cell]{xy}
 \SelectTips{eu}{10} \UseTips
 \UseAllTwocells

\DeclareMathAlphabet\EuR{U}{eur}{m}{n}
\SetMathAlphabet\EuR{bold}{U}{eur}{b}{n}

\newcommand{\calh}{\mathcal{H}}

\newcommand{\call}{\mathcal{L}}

\newcommand{\IC}{{\mathbb C}}

\newcommand{\IQ}{{\mathbb Q}}
\newcommand{\IR}{{\mathbb R}}

\newcommand{\IZ}{{\mathbb Z}}

\newcommand{\Alt}{\operatorname{Alt}}
\newcommand{\an}{\operatorname{an}}

\newcommand{\coker}{\operatorname{coker}}
\newcommand{\clos}{\operatorname{clos}}

\newcommand{\cone}{\operatorname{cone}}

\newcommand{\dR}{\operatorname{dR}}

\newcommand{\harm}{\operatorname{harm}}

\newcommand{\id}{\operatorname{id}}

\newcommand{\im}{\operatorname{im}}

\newcommand{\pd}{\operatorname{pd}}
\newcommand{\pr}{\operatorname{pr}}
\newcommand{\Real}{\operatorname{Real}}
\newcommand{\Rep}{\operatorname{Rep}}

\newcommand{\Riem}{\operatorname{Riem}}

\newcommand{\sing}{\operatorname{sing}}
\newcommand{\sign}{\operatorname{sign}}

\newcommand{\spec}{\operatorname{spec}}
\newcommand{\topo}{\operatorname{top}}

\newcommand{\tors}{\operatorname{tors}}
\newcommand{\tr}{\operatorname{tr}}

\newcommand{\vol}{\operatorname{vol}}

\newcommand{\dvol}{\,d\!\vol}

\DeclareTextFontCommand{\textcyr}{\cyr}

\def\cprime{\char126}

\newcommand{\higherlim}[3]{{\setbox1=\hbox{\rm lim}
        \setbox2=\hbox to \wd1{\leftarrowfill} \ht2=0pt \dp2=-1pt
        \mathop{\vtop{\baselineskip=5pt\box1\box2}}
        _{#1}}^{#2}#3}

\newcommand{\version}[1]                       
{\begin{center} last edited on #1\\
last compiled on \today\\
name of texfile: \jobname
\end{center}
}

\newcounter{commentcounter}

\begin{document}

\maketitle


\thispagestyle{empty}

\begin{abstract}

The article consists of a survey on analytic and topological torsion. Analytic
torsion is defined in terms of the spectrum of the analytic Laplace
operator on a Riemannian manifold, whereas  topological torsion is
defined in terms of a triangulation. The celebrated theorem of Cheeger
and M\"uller identifies these two notions for closed Riemannian
manifolds. We also deal with manifolds with boundary and with isometric actions
of finite groups.  The basic theme is to extract topological invariants from the spectrum of
the analytic Laplace operator on a Riemannian manifold.

\end{abstract}

\setcounter{section}{-1}


\typeout{------------------------------- Section 0: Introduction   --------------------------------}

\section{Introduction}

When I was asked to write a contribution to a book in honor of Bernhard Riemann, I was on
one side flattered, but on the other side also scared. Although Riemann has done so much
foundational and seminal work in many areas, there was no obvious topic, where I may have
something to say and on which Riemann has worked. Moreover, I am  obviously not an expert on the 
history of mathematics.

After some thought I decided to choose as topic analytic and topological torsion. This is
an interesting example for an interaction between analysis and topology and this is seems
to be a theme, in which Riemann was interested..  The goal  is to extract topological
invariants from the spectrum of the analytic Laplace operator on a Riemannian manifold.

Finally I had to decide on the structure of the paper and for whom it should be written.
A technical paper on latest results did not seem to be appropriate. So I decided to tell
the story how one can come from elementary considerations about linear algebra of
finite-dimensional Hilbert spaces and elementary invariants such as dimension, trace, and
determinant to topological notions, which are in general easy, and then to their analytic
counterparts, which are in general much more difficult.  Hopefully the first sections are
comprehensible even for graduate students and present some important tools and notions, which can be
transferred to the analytic setting with some effort. Moreover, this transition explains
the basic ideas underlying the analytic notions. For an advanced mathematician, who is not
an expert on analytic or topological torsion, it may be interesting to see how this
interaction between analysis and topology is developed and what its impact is.  We tried to
keep the exposition as simple as possible to ensure that the paper is accessible. 
This also means that for an expert on analytic and topological torsion this article will contain 
no new information.

Here is a brief summary of the contents of this paper.

In the first section we recall in the framework of linear maps between finite-dimensional
Hilbert spaces basic notions such as the trace, the determinant and the spectrum.  We will
rewrite the classical notion of a determinant in terms of the Zeta-function and the
spectral density function. The point will be that in this new form they
can be extended to the analytic setting, where one has to deal with infinite-dimensional Hilbert spaces.
This is not possible if one sticks to the classical definitions.

In the second section we consider finite Hilbert chain complexes, which are chain
complexes of finite-dimensional Hilbert spaces for which only finitely many chain modules
are not zero. For those we can define Betti numbers and torsion invariants and give an
elementary ``baby'' version of the Hodge de Rham decomposition.

In the third section we pass to analysis. Our first interaction between analysis and
topology will be presented by the de Rham Theorem.  Then we will explain the Hodge-de Rham
Theorem which relates the singular cohomology of a closed Riemannian manifold to the space
of harmonic forms.

In the fourth section topological torsion is defined by considering cellular chain
complexes of finite $CW$-complexes or closed Riemannian manifolds. It can be written in
terms of the combinatorial Laplace operator in an elementary fashion except that one has
to correct the Hilbert space structure on the homology using the isomorphisms of the third
section.

The fifth section is devoted to analytic torsion. Its definition is rather complicated,
but it should become clear what the idea behind the definition is, in view of the definition of
the topological torsion. We will explain that topological and analytical torsion agree
for closed Riemannian manifolds. If the compact Riemannian manifold has boundary, then a
correction term based on the Euler characteristic of the boundary is needed.

In the sixth section  the results of the fifth section are extended to compact Riemannian manifolds
with an isometric action of a finite group.  Here a new phenomenon occurs, namely a third torsion
invariant, the Poincar\'e torsion, comes into play.

In the seventh section we give a very brief overview over the literature about analytic and
topological torsion and its generalization to the $L^2$-setting.

\subsection*{Acknowledgments}
The paper is financially supported by the Leibniz-Award of the author granted by
the Deutsche Forschungsgemeinschaft {DFG}.


\tableofcontents


\typeout{------------------------- Section 1: Operators of finite-dimensional Hilbert spaces  ---------------------}

\section{Operators of finite-dimensional Hilbert spaces}
\label{sec:Operators_of_finite-dimensional_Hilbert_spaces}

In this section we review some well-known concepts about a linear map between
finite-dimensional (real) Hilbert spaces such as its determinant, its trace,
and its spectrum.  All of the material presented in this section is accessible to a student in his second year.
Often key ideas can easily be seen and illustrated in this elementary context. Moreover,
we will sometimes rewrite a well-known notion in a fashion which will later allow us to
extend it to more general situations.


\subsection{Linear maps between finite-dimensional vector spaces}
\label{subsec:Linear_maps_between_finite-dimensional_vector_spaces}

Let $f \colon V \to W$ be a linear map of finite-dimensional (real) vector spaces. Recall that every finite-dimensional
vector space $V$ carries a unique topology which is characterized by the property that  any linear isomorphism
$f \colon \IR^n \xrightarrow{\cong} V$  is a homeo\-mor\-phism. This definition makes sense since
any linear automorphism of $\IR^n$ is a homeomorphism. In particular any linear map $f \colon V \to W$ 
of finite-dimensional vector spaces is an \emph{operator}, i.e., a continuous linear map. 

We can assign to an endomorphisms $f \colon V \to V$ of  a finite-dimensional vector space $V$ two basic invariants,
its trace  and its determinant,  as follows. If we write
$V^* = \hom_{\IR}(V,\IR)$, then there are canonical linear maps
\begin{eqnarray*}
\alpha \colon V^* \otimes V 
& \to & 
\hom_{\IR}(V,V), \quad \phi \otimes v \mapsto\left(w \mapsto \phi(w) \cdot v\right);
\\
\beta \colon V^* \otimes V 
& \xrightarrow{\cong} & 
\hom_{\IR}(V,V), \quad \phi \otimes v \mapsto \phi(v).
\end{eqnarray*}
The first one is an isomorphism. Hence we can define the trace map to be the composite
\[
\tr \colon \hom_{\IR}(V,V) \xrightarrow{\alpha^{-1}} V^* \otimes V \xrightarrow{\beta} \IR,
\]
and the \emph{trace of $f$} 
\begin{equation} 
\tr(f) \in \IR
\label{trace_of_f_finite_dimensional}
\end{equation}
to be the image of $f$ under this linear map. 
The trace has the basic properties that $\tr(g \circ f) = \tr(f \circ g)$ holds for linear
maps $f \colon U \to V$ and $g \to V \to W$, it is linear, i.e., $\tr(r \cdot f + s
\cdot g) = r \cdot \tr(f) + s \cdot \tr(g)$, and $\tr(\id_{\IR} \colon \IR \to \IR) = 1$.  We
leave it to the reader to check that these three properties determine the trace uniquely.

If $n$ is the dimension of $V$, the vector space $\Alt^n(V)$ of alternating $n$-forms 
$V\times V \times \cdots \times V \to \IR$ has dimension one. An endomorphism $f \colon V \to V$ 
induces an endomorphism $\Alt^n(f) \colon \Alt^n(V) \to \Alt^n(V)$.  Hence there is
precisely one real number $r$ such that $\Alt^n(f) = r \cdot \id_{\Alt^n(V)}$ and we define the
\emph{determinant}  of $f$ to be $r$, or, equivalently, by the equation
\begin{equation}
\det(f) \cdot \id_{\Alt^n(f)} = \Alt^n(f).
\label{determinant_finite-dimensional}
\end{equation}
The determinant has the properties that
$\det(g \circ f) = \det(g) \cdot \det(f)$ holds for endomorphisms $f,g \colon V \to V$,
for any commutative diagram with exact rows
\[\xymatrix{
0 \ar[r]
&
U \ar[r]^i \ar[d]^f
&
V \ar[r]^p \ar[d]^g
&
W \ar[r] \ar[d]^h
&
0
\\
0 \ar[r]
&
U \ar[r]^i 
&
V \ar[r]^p
&
W \ar[r] 
&
0
}
\]
we get $\det(g) = \det(f) \cdot \det(h)$ and $\det(\id_{\IR}) = 1$. 
We leave it to the reader to check that these three properties determine the determinant uniquely.

Notice that all of our definitions are intrinsic, we do not use bases. Of course if we
choose a basis $\{b_1, b_2, \ldots ,b_n\}$ for $V$ and let $A$ be the $(n,n)$-matrix
describing $f$ with respect to this basis, then we get back the standard definitions in
terms of matrices
\begin{eqnarray*}
\tr(f) 
& = & 
\sum_{i =1}^n a_{i,i};
\\
\det(f) 
& = & 
\prod_{\sigma \in  S_n} \sign(\sigma) \cdot \prod_{i=1}^n a_{i,\sigma(i)}.
\end{eqnarray*}


\subsection{Linear maps between finite-dimensional Hilbert spaces}
\label{subsec:Linear_maps_between_finite-dimensional_Hilbert_spaces}

Now we consider finite-dimensional Hilbert spaces, i.e., finite-dimensional vector spaces
with an inner product.  Notice that we do not have to require that $V$ is complete with
respect to the induced norm, this is automatically fulfilled.  Let $f \colon U \to V$ be
a linear map. Its \emph{adjoint} is the linear map $f^* \colon V \to W$ uniquely
determined by the property that $\langle f(v),w \rangle_W = \langle v,f^*(w) \rangle_V$
holds for all $v \in V$ and $w \in W$. If we choose orthonormal basis for $V$ and $W$ and
let $A(f)$ and $A(f^*)$ be the matrices describing $f$ and $f^*$, then $A(f)^*$ is the
transpose of $A(f)$.  We call an endomorphism $f \colon V \to V$
\emph{selfadjoint} if and only if $f^* = f$.  This is equivalent to the condition that
$A(f)$ is symmetric. We call an endomorphism $f \colon V \to V$ \emph{positive} if 
$\langle f(v),v \rangle \ge 0$ holds for all $v \in V$. This is equivalent to the existence of a
linear map $g \colon V \to V$ with $f = g^*g$. In particular every positive linear endomorphism  is
selfadjoint.

The following version of a determinant will be of importance for us.  Let $f \colon V \to W$ 
be a linear map of finite-dimensional Hilbert spaces, where $V$ and $W$ may be
different.  Then $f^*f \colon V \to V$ induces an automorphism $(f^*f)^{\perp} \colon
\ker(f^*f)^{\perp} \xrightarrow{\cong} \ker(f^*f)^{\perp}$, where $\ker(f^*f)^{\perp}$ is
the orthogonal complement of $\ker(f^*f)$ in $U$.  Define
\begin{equation}
{\det}^{\perp}(f) 
:= 
\begin{cases} 
\sqrt{\det\bigl((f^*f)^{\perp} \colon \ker(f^*f)^{\perp} \xrightarrow{\cong} \ker(f^*f)^{\perp}\bigr)}
& \text{if}\; f \not= 0
\\
1 & \text{if}\; f = 0.
\end{cases}
\label{det_perp(f)-finite-dimensional}
\end{equation}

The proof of the following elementary lemma is left to the reader, 
or consult~\cite[Theorem~3.14 on page~128 and Lemma~3.15 on page~129]{Lueck(2002)}.

\begin{lemma} \label{lem:main_properties_of_det_perp}\
\begin{enumerate}
\item \label{lem:main_properties_of_det_perp:auto} If $f \colon V \to V$ is a linear
  automorphism of a finite-dimensional Hilbert space, then ${\det}^{\perp}(f) = |\det(f)|$
  for $\det(f)$ the classical determinant;

\item \label{lem:main_properties_of_det_perp:composition} Let $f\colon U \to V$ and
  $g\colon V \to W$ be linear maps of  finite-dimensional Hilbert spaces such that $f$
  is surjective and $g$ is injective. Then
\[
{\det}^{\perp}(g \circ f)  =  {\det}^{\perp}(f) \cdot {\det}^{\perp}(g);
\]

\item
\label{lem:main_properties_of_det_perp:additivity}
Let $f_1\colon  U_1 \to V_1$, $f_2\colon  U_2 \to V_2$ and
$f_3\colon  U_2 \to V_1$ be linear maps of finite-dimensional Hilbert spaces
such that  $f_1$ is surjective  and $f_2$ is injective. Then
\[
{\det}^{\perp}\begin{pmatrix} f_1 & f_3 \\ 0 & f_2\end{pmatrix}
 = 
{\det}^{\perp}(f_1) \cdot {\det}^{\perp}(f_2);
\]

\item
\label{lem:main_properties_of_det_perp:direct_sums}
Let $f_1\colon U_1 \to V_1$ and $f_2\colon U_2 \to V_2$ be linear maps of finite-dimensional
Hilbert spaces.  Then
\[
{\det}^{\perp}(f_1 \oplus f_2) 
 = 
{\det}^{\perp}(f_1) \cdot {\det}^{\perp}(f_2);
\]

\item \label{lem:main_properties_of_det_perp:det(f)_is_det(f_ast)} 
Let $f\colon U \to V$  be a linear map of finite-dimensional Hilbert spaces. Then
\[
{\det}^{\perp}(f) = {\det}^{\perp}(f^*) =\sqrt{{\det}^{\perp}(f^*f)} = \sqrt{{\det}^{\perp}(ff^*)}.
\]
\end{enumerate}
\end{lemma}


\subsection{The spectrum and the spectral density function}
\label{subsec:The_spectrum_and_in_the_spectral_density_function}

If one is only interested in finite-dimensional Hilbert spaces and in Betti numbers or
torsion invariants for finite $CW$-complexes, one does not need the following material of
the remainder of this Section~\ref{sec:Operators_of_finite-dimensional_Hilbert_spaces}. 
However, we will now lay the foundations to extend these
invariants to the analytic setting or to the $L^2$-setting, where the Hilbert spaces are
not finite-dimensional anymore.

The spectrum $\spec(f)$ of a selfadjoint operator $f \colon V \to V$ of finite-dimensional
Hilbert spaces consists of the set of eigenvalues $\lambda$ of $f$, i.e., real numbers
$\lambda$ for which there exists $v \in V$ with $v \not= 0$ and $f(v) = \lambda \cdot
v$. The multiplicity $\mu(f)(\lambda)$ of an eigenvalue $\lambda$ is the dimension of its
eigenspace 
\[
E_{\lambda}(f) := \{v \in V \mid f(v) = \lambda \cdot v\}.
\]
If $\lambda \in \IR$ is not an eigenvalue,
we put $\mu(f)(\lambda) = 0$. An elementary but basic result in
linear algebra says that for a selfadjoint linear map $f \colon V \to V$ there exists an
orthonormal basis of eigenvectors of $V$. A selfadjoint linear endomorphism is positive if
$\lambda \ge 0$ holds for each eigenvalue $\lambda$. 

Next we introduce for a linear map $f \colon U \to V$ of finite-dimensional Hilbert spaces
its \emph{spectral density function}
\begin{equation}
F(f) \colon [0,\infty) \to [0,\infty).
\label{spectral_density_function_finite-dimensional_Hilbert_spaces}
\end{equation}
It is defined as the following right continuous step function. Its value at zero is the
dimension of the kernel of $f^*f$.  Notice that $\ker(f^*f) = \ker(f)$ since $v \in
\ker(f^*f)$ implies $0 = \langle f^*f(v),v) \rangle = \langle f(v),f(v) \rangle$ and hence
$f(v) = 0$.  The jumps of the step function happen exactly at the square roots of the
eigenvalues of $f^*f$ and the height of the jump is the multiplicity $\mu(f^*f)(\lambda)$ of the
eigenvalue. There is a number $C \ge 0$ such that $F(f)(\lambda) = \dim(V)$ holds for
all $\lambda \ge C$, for instance, take $C$ to be the square root of the largest eigenvalue of $f^*f$.
Obviously $f$ is injective if and only if $F(f)(0) = 0$.

Suppose that $f$ is already a positive operator $f \colon V \to V$. Then $f^*f$ is $f^2$.
Moreover, $F(f)$ has the dimension of $\ker(f)$ as value at zero and the step function jumps
exactly at those $\lambda \in \IR$ which are eigenvalues of $f$ and the height of the jump 
is $\mu(f)(\lambda)$.

One can also define the spectral density function of a linear map $f \colon V \to W$ of
finite-dimensional Hilbert spaces without referring to eigenvalues in a more intrinsic
way as follows.  Let $\call(f,\lambda)$ be the set of linear subspaces $L \subseteq V$
such that $||f(v)|| \le \lambda \cdot ||v||$ holds for every $v \in L$. Then we
from~\cite[Lemma~2.3 on page~74]{Lueck(2002)}
\begin{equation}
F(f)(\lambda) = \sup\{\dim(L) \mid L \in \call(f,\lambda)\}.
\label{F(f)(lambda)_n_terms_of_call(f,lambda)-finite-dimensional}
\end{equation}


\subsection{Rewriting determinants}
\label{subsec:Rewriting_determinants}

The following  formula will be of central interest for us. 
Let $f \colon V \to W$ be a linear map of finite-dimensional Hilbert spaces.
Notice that ${\det}^{\perp}(f) > 0$ so that we can consider the real number $\ln({\det}^{\perp}(f))$.
The formula
\begin{equation}
\ln({\det}^{\perp}(f)) 
 =  
\frac{1}{2} \cdot \sum_{\substack{\lambda \in \spec(f^*f), \\ \lambda > 0}}
\mu(f^*f)(\lambda) \cdot \ln(\lambda)
\label{ln(det_perp)_and_sum_of_logarithms_of-eigenvalues_finite_dimensional}
\end{equation}
is a direct consequence of the fact that we have orthogonal decompositions
\begin{eqnarray*}
V
& = & 
\bigoplus_{\lambda \in \spec(f^*f)} E_{\lambda}(f^*f);
\\
\ker(f^*f)^{\perp} 
& = & 
\bigoplus_{\substack{\lambda \in \spec(f^*f), \\ \lambda > 0}} E_{\lambda}(f^*f).
\end{eqnarray*}
We can use this orthogonal decomposition to define a new linear automorphism of $V$ by
\begin{eqnarray}
\ln((f^*f)^{\perp}) := \bigoplus_{\substack{\lambda \in \spec(f^*f), \\ \lambda > 0}} \;\ln(\lambda) \cdot \id_{E_{\lambda}(f^*f)}
\colon \ker(f^*f)^{\perp} \to \ker(f^*f)^{\perp}.
\label{ln(f_astf}
\end{eqnarray}
Then we can rephrase~\eqref{ln(det_perp)_and_sum_of_logarithms_of-eigenvalues_finite_dimensional}
as 
\begin{equation}
\ln({\det}^{\perp}(f)) 
 =  
\frac{1}{2} \cdot  \tr\bigl(\ln((f^*f)^{\perp})\bigr).
\label{ln(det_perp)_and_trace_of_logarithm__finite_dimensional}
\end{equation}

The following observation will be the key to define determinants also for operators
between not necessarily finite-dimensional Hilbert spaces, for instance for the analytic 
Laplace operator acting on smooth $p$-forms for a closed Riemannian manifold.
Namely, we define a holomorphic function $\zeta_f\colon \IC \to \IC$ by
\begin{equation}
\zeta_f(s) = \sum_{\substack{\lambda \in \spec(f^*f), \\ \lambda > 0}}
\mu(f^*f)(\lambda) \cdot \lambda^{-s},
\label{zeta(f)_finite-dimensional}
\end{equation}
Then one easily checks
using~\eqref{ln(det_perp)_and_sum_of_logarithms_of-eigenvalues_finite_dimensional}
\begin{eqnarray}
- \ln({\det}^{\perp}(f)) 
 & = & 
- \frac{1}{2} \cdot  \sum_{\substack{\lambda \in \spec(f^*f), \\ \lambda > 0}}
\mu(f^*f)(\lambda) \cdot \ln(\lambda)
\label{det_and_zeta_fin_dim}
\\
& = & 
\frac{1}{2} \cdot  \sum_{\substack{\lambda \in \spec(f^*f), \\ \lambda > 0}}
\mu(f^*f)(\lambda) \cdot \left.\frac{d}{ds}\right|_{s = 0}  \lambda^{-s}
\nonumber
\\
& = & 
\frac{1}{2} \cdot  \left.\frac{d}{ds}\right|_{s = 0} \zeta_f.
\nonumber
\end{eqnarray}

In order to extend the notion of ${\det}^{\perp}(f)$ in the $L^2$-setting to the Fuglede-Kadison determinant,
it is useful to rewrite the quantity $\ln({\det}^{\perp}(f))$ in terms of an integral with respect to  measure coming from
the spectral density function as follows.

Recall that $F(f)$ is a monotone non-decreasing right-continuous
function.  Denote by $dF(f)$ the measure on the Borel $\sigma$-algebra on $\IR$
which is uniquely determined by its values on the half open intervals
$(a,b]$ for $a < b$ by  $dF(f)((a,b])  = F(f)(b) - F(f)(a)$.
The measure of the one point set $\{a\}$
is $\lim_{x\to 0+} F(f)(a)-F(f)(a-x)$ and is zero if and only if
$F(f)$ is left-continuous in $a$. We will use here and in the sequel the convention that
$\int_a^b$, $\int_{a+}^b$, $\int_a^{\infty}$ and $\int_{a+}^{\infty}$
respectively
means integration over the interval $[a,b]$, $(a,b]$, $[a,\infty)$ and
$(a,\infty)$ respectively. An easy computation 
using~\eqref{ln(det_perp)_and_sum_of_logarithms_of-eigenvalues_finite_dimensional}
shows
\begin{eqnarray}
\ln({\det}^{\perp}(f^*f))
& = & 
\int_{0+}^{\infty} \ln(\lambda) \;dF(f).
\label{ln(det_perp_and_int_ln(lambda)df)_finite-dimensional_case}
\end{eqnarray}
Elementary integration theory shows that we get for $d\lambda$ the standard Lebesgue measure
and any $a \ge  \dim(U)$
\begin{equation}
\int_{0+}^{\infty} \ln(\lambda) \;dF
 = 
\ln(a) \cdot (F(a) - F(0)) -\int_{0+}^a \frac{F(f)(\lambda) - F(f)(0)}{\lambda} \; d\lambda.
\label{int_ln(lambda)df)_and_int_in_terms_of_Lebesgue_measure_finite-dimensional_case}
\end{equation}


\typeout{------------ Section 2: -------------}

\section{Finite Hilbert chain complexes}
\label{sec:Finite_Hilbert_chain_complexes}

Having in mind the cellular chain complex of a finite $CW$-complex, we want to consider
now finite Hilbert chain complexes. A \emph{finite Hilbert chain complex} 
$C_* = (C_*,c_*)$ consists of a collection of finite-dimensional Hilbert spaces $C_n$ and linear
maps $c_n \colon C_n \to C_{n-1}$ for $n \in \IZ$ such that $c_n \circ c_{n+1} = 0$ holds
for all $n \in \IZ$ and there exists a natural number $N$ with $c_n = 0$ for $|n| > N$. A
\emph{chain map of finite Hilbert chain complexes} $f_* \colon C_* \to D_*$ is a
collection of linear maps $f_n \colon C_n \to D_n$ for $n \in \IZ$ such that $d_n \circ
f_n = f_{n-1} \circ c_n$ holds for all $n \in \IZ$. (We do not require that the maps $f_n$
are compatible with the Hilbert space structures.) It is obvious what a chain homotopy
and a chain homotopy equivalence of finite Hilbert chain complexes means.  The homology
$H_n(C_*)$ is the Hilbert space $\ker(c_n)/\im(c_{n+1})$, where $\ker(c_n)$ is equipped
with the Hilbert space structure coming from $C_n$ and $\ker(c_n)/\im(c_{n+1})$ inherits
the quotient Hilbert space structure. Define the \emph{$n$-th Laplace operator} 
\begin{equation}
\Delta_n = c_n^* \circ c_n + c_{n+1} \circ c_{n+1}^* \colon C_n \to C_n.
\label{Delta_n_finite_Hilbert_chain_complexes}
\end{equation}

The importance of the following notions cannot be underestimated.

\begin{definition}[Betti numbers and torsion of a finite Hilbert chain complex]
\label{def:Betti_numbers_and_torsion_of_a_finite_Hilbert_chain_complex}
Let $C_*$ be a finite Hilbert chain complex. 

Define its \emph{$n$-th Betti number}
\[
b_n(C_*) := \dim(H_n(C_*)) \quad \in \IZ_{\ge 0}.
\]
Define its \emph{torsion}
\[
\rho(C_*) := - \sum_{n \in \IZ} (-1)^n \cdot \ln\bigl({\det}^{\perp}(c_n)\bigr) \quad \in \IR,
\]
where $\det^{\perp}$ has been introduced in~\eqref{det_perp(f)-finite-dimensional}.
\end{definition}


\subsection{Betti numbers}
\label{subsec:Betti_numbers}

Next we relate these notions to the Laplace operator.
The following result is a ``baby''-version of the Hodge-de Rham Theorem,
see Subsection~\ref{subsec:The_Hodge_deRham_Theorem}.
In the sequel we equip $\ker(\Delta_n) \subseteq C_n$ with the Hilbert space
structure induced from the given one on $C_n$.

\begin{lemma} \label{lem:Hodge_decomposition}
Let $C_*$ be a finite Hilbert chain complex. Then we get for all $n \in \IZ$
\[
\ker(\Delta_n) = \ker(c_n) \cap \im(c_{n+1})^{\perp},
\]
and an orthogonal decomposition
\[
C_n = 
\im(c_n^*) \oplus\im(c_{n+1}) \oplus \ker(\Delta_n).
\]
In particular the obvious composite
\[
\ker(\Delta_n) \to \ker(c_n) \to H_n(C_n)
\]
is an isometric isomorphism of Hilbert spaces and we get
\[
b_n(C_*) = \dim(\ker(\Delta_n)).
\]
\end{lemma}
\begin{proof}
Consider $v \in V$. We compute
\begin{eqnarray*}
\langle c_n(v),c_n(v) \rangle  + \langle c_{n+1}^*(v),c_{n+1}^*(v)\rangle
& = & 
\langle c_n^* \circ c_n(v), v \rangle + \langle c_{n+1} \circ c_{n+1}^*(v), v \rangle 
\\
& = & 
\langle c_n^* \circ c_n(v)  + c_{n+1} \circ c_{n+1}^*(v), v \rangle 
\\
& = & 
\langle \Delta_n(v), v \rangle.
\end{eqnarray*}
Hence we get for $v \in V$ that $\Delta_n(v) = 0$ is equivalent to $c_n(v) = c_{n+1}^*(v) = 0$.
This shows $\ker(\Delta_n) = \ker(c_n) \cap \ker(c_{n+1}^*) =  \ker(c_n) \cap \im(c_{n+1})^{\perp}$.
The other claims are now direct consequences.
\end{proof}

\begin{remark}[Homotopy invariance of $\dim(\ker(\Delta_n)$]
  \label{rem:Homotopy_invariance_of_dim(ker(Delta_n)_combinatorial}
  Notice the following fundamental consequence of Lemma~\ref{lem:Hodge_decomposition}
  that $\dim(\ker(\Delta_n))$ depends only on the chain homotopy type of $C_*$ and is in
  particular independent of the Hilbert space structure on $C_*$ since for a chain
  homotopy equivalence $f_* \colon C_* \to D_*$ we obtain an isomorphism 
  $H_n(f_*) \colon  H_n(C_*) \to H_n(D_*)$ and hence the equality $b_n(C_*) = b_n(D_*)$.
  
 Of course the spectrum of the Laplace operator $\Delta_n$ does depend on the Hilbert space structure, but
 a part of it, namely,  the multiplicity of the eigenvalue $0$, which is just $\dim(\ker(\Delta_n))$, 
 depends only on the homotopy type of $C_*$. 
\end{remark}

\begin{remark}[Heat operator]\label{rem:heat_operator_hilbert}
One can assign to the Laplace operator $\Delta_n \colon C_n \to C_n$ its \emph{heat operator}
 $e^{-t\Delta_n}$. It is defined analogously to $\ln(f^*f)$, see~\eqref{ln(f_astf}, 
but now each eigenvalue $\lambda$ of $\Delta_n$ transforms to the  eigenvalue $e^{-t\lambda}$. 

Then we obviously get
\[
b_n(C_*) = \lim_{t \to \infty} \tr(e^{-t \Delta_n}).
\]
\end{remark}

\subsection{Torsion for finite Hilbert chain complexes}
\label{subsec:Torsion_for_finite_Hilbert_chain_complexes}

 The situation with torsion is more complicated, but in some sense similar, as we explain next.
First of all one can rewrite torsion in terms of the Laplace operator.

\begin{lemma} \label{lem:torsion_in_terms_of_Laplacian}
If $C_*$ is a finite Hilbert chain complex, then we get
\[
\rho(C_*) = -\frac{1}{2} \cdot \sum_{n \in \IZ} (-1)^n \cdot n \cdot \ln\bigl({\det}^{\perp}(\Delta_n)\bigr).
\]
\end{lemma}
\begin{proof}
From  Lemma~\ref{lem:Hodge_decomposition} we obtain
an orthogonal decomposition
\begin{eqnarray*}
C_n
& = & \ker(c_n)^{\perp} \oplus\im(c_{n+1}) \oplus \ker(\Delta_n);
\\
\Delta_n
& = & ((c_n^{\perp})^*\circ c_n^{\perp})
\oplus (c_{n+1}^{\perp} \circ (c_{n+1}^{\perp})^*) \oplus 0,
\end{eqnarray*}
where $c_n^{\perp}\colon  \ker(c_n)^{\perp} \to\im(c_n)$ is the 
weak isomorphism induced by $c_n$. Now we compute using
Lemma~\ref{lem:main_properties_of_det_perp}
\begin{eqnarray*}
\lefteqn{
- \frac{1}{2}\cdot \sum_{n \in \IZ} (-1)^n \cdot n \cdot \ln({\det}^{\perp}(\Delta_n))}
& &
\\
&  = &
- \frac{1}{2}\cdot \sum_{n \in \IZ} (-1)^n \cdot n \cdot
\ln\bigl({\det}^{\perp}\bigl(((c_n^{\perp})^*\circ c_n^{\perp})
\oplus (c_{n+1}^{\perp} \circ (c_{n+1}^{\perp})^*)
\oplus 0\bigr)\bigr)
\\
&  = &
- \frac{1}{2}\cdot \sum_{n \in \IZ} (-1)^n \cdot n \cdot
\left(\ln\bigl({\det}^{\perp}((c_n^{\perp})^*\circ c_n^{\perp})\bigr)\right.
\\
& & \hspace{30mm}
+ \left.\ln\bigl({\det}^{\perp}(c_{n+1}^{\perp} \circ (c_{n+1}^{\perp})^*)\bigr)
+ \ln({\det}^{\perp}(0))\right)
\\
&   = &
- \frac{1}{2}\cdot \sum_{n \in \IZ} (-1)^n \cdot n \cdot
\left(2 \cdot \ln\bigl({\det}^{\perp}(c_n)\bigr)
+ 2 \cdot \ln\bigl({\det}^{\perp}(c_{n+1})\bigr)\right)
\\
&   = &
- \sum_{n \in \IZ} (-1)^n \cdot
\ln\bigl({\det}^{\perp}(c_n)\bigr). 
\end{eqnarray*}
\end{proof}

A basic property of the torsion is additivity, whose proof can be found 
in~\cite[Theorem~3.35~(1) on page~142]{Lueck(2002)}.

\begin{lemma}
\label{lem:additivity_of-torsion_fin_dim}
Consider the short exact sequence of finite Hilbert chain complexes $0 \to C_*
\xrightarrow{i_*} D_* \xrightarrow{p_*} E_* \to 0$.  For each $n \in \IZ$ we obtain a
finite Hilbert chain complex $E[n]_*$ concentrated in dimension $0$, $1$, and $2$ which is
given there by $C_n \xrightarrow{i_n} D_n \xrightarrow{p_n} E_n$. The long exact homology
sequence associated to $0 \to C_* \xrightarrow{i_*} D_* \xrightarrow{p_*} E_* \to 0$ can
be viewed as finite Hilbert chain complex denoted by $LHS_*$.

Then we get 
\[
\rho(C_*) - \rho(D_*) + \rho(E_*) =
\left(\sum_{n \in \IZ} (-1)^n \cdot \rho(E[n]_*)\right) - \rho(LHS_*).
\]
\end{lemma}

Let $f \colon C_* \to D_*$ be a chain map of finite Hilbert chain complexes. Let $\cone(f_*)$ 
be its mapping cone whose $n$-th differential is given by
\[
\begin{pmatrix} 
-c_{n-1} 
& 
0 
\\ 
f_{n-1} 
& 
d_{n}
\end{pmatrix}
\colon C_{n-1} \oplus D_n \to C_{n-1} \oplus D_{n-1}.
\]
Define the torsion of $f_*$ by
\begin{equation}
\tau(f_*) := \rho(\cone(f_*)).
\label{tau(f_ast)_fin_dim}
\end{equation}

\begin{lemma} \label{lem:homotopy_invariant_of_torsion_fin_dim}
Let $f \colon C_* \to D_*$ be a chain homotopy equivalence of finite Hilbert chain complexes. 
Then  we get
\[
\tau(f_*) = \rho(D_*) - \rho(C_*) + \sum_{n \in \IZ} (-1)^n \cdot \ln\bigl({\det}^{\perp}(H_n(f_*))\bigr).
\]
\end{lemma}
\begin{proof} 
This follows from Lemma~\ref{lem:additivity_of-torsion_fin_dim} applied to the
canonical short exact sequence $0 \to D_* \to \cone(f_*) \to \Sigma C_* \to 0$
using the fact that $H_n(\cone(f_*)) = 0$ holds for $n \in \IZ$.
\end{proof}

\begin{lemma}
\label{lem:isos_of_chain_complexes-fin_dim}
Let $f_* \colon C_* \to D_*$ be a chain map of finite contractible Hilbert chain complexes
such that $f_n$ is bijective for each $n \in \IZ$. Then
\[
\rho(D_*) - \rho(C_*) = \sum_{n \in \IZ} (-1)^n \cdot \ln\bigl({\det}^{\perp}(f_n)\bigr).
\]
\end{lemma}
\begin{proof}
Because of Lemma~\ref{lem:homotopy_invariant_of_torsion_fin_dim} it suffices to show
\[
\rho(\cone(f_*)) = \sum_{n \in \IZ} (-1)^n \cdot \ln\bigl({\det}^{\perp}(f_n)\bigr).
\]
This is done by induction over the length of $C_*$ which is the supremum 
$\{m-n \mid C_m \not= 0, C_n \not= 0\}$.  The induction step, when the length is less or equal 
to one, follows directly from the definitions.  The induction step is done as follows.  Let $m$ be
the largest integer with $C_m \not= 0$. Let $C_*|_{m-1}$ obtained from $C_*$ by putting $C_m =0$ 
and leaving the rest. Obviously $f_* \colon C_* \to D_*$ induces a chain isomorphism
$f_*|_{m-1} \colon C_*|_{m-1} \to D_*|_{m-1}$. Let $m[C_*]$ be the chain complex
concentrated in dimension $m$ whose $m$-th chain module is $C_m$. Obviously $f_*$ induces
a chain isomorphism $m[f_*] \colon m[C_*] \to m[D_*]$. We have the obvious short exact
sequence of finite contractible Hilbert chain complexes 
$0 \to \cone(m[f_*]) \to \cone(f_*) \to \cone([f_*|_{m-1}) \to 0$.  
Lemma~\ref{lem:additivity_of-torsion_fin_dim} implies
\[
\rho(\cone(f_*) ) = \rho(\cone(m[f_*])) + \rho\bigl(\cone(f_*|_{m-1})\bigr).
\]
The induction hypothesis applies to $m[C_*]$ and $C_*|_{m-1}$ and thus we have
\begin{eqnarray*}
\rho(\cone(m[f_*])) 
& = & 
(-1)^m \cdot \ln\bigl({\det}^{\perp}(f_m)\bigr);
\\
\rho\bigl(\cone(f_*|_{m-1})\bigr)
& = & 
\sum_{n \in \IZ, n \not= m } (-1)^n \cdot \ln\bigl({\det}^{\perp}(f_n)\bigr).
\end{eqnarray*}
This finishes the proof Lemma~\ref{lem:isos_of_chain_complexes-fin_dim}.
\end{proof}


\subsection{Torsion for finite based free $\IZ$-chain complexes}
\label{subsec:Torsion_for_finite_based_free_Z-chain_complexes}

Let $C_*$ be a finite free $\IZ$-chain complex, i.e., a $\IZ$-chain complex whose chain
modules are all finitely generated free abelian groups and for which there exists a natural
number $N$ such that $c_n = 0$ for $|n| > N$.  Given a finitely generated free
$\IZ$-module $M$, we call two $\IZ$-bases $B = \{b_1, b_2, \ldots, b_n\}$ and $B' = \{b_1',
b_2', \ldots, b_n'\}$ equivalent if there exists a permutation $\sigma \in S_n$ and elements 
$\epsilon_i \in \{\pm 1\}$ for $i = 1,2, \ldots, n$ such that $b'_{\sigma(i)} = \epsilon_i \cdot b_i$
holds for $i =1,2, \ldots, n$.  A $\IZ$-basis $B = \{b_1, b_2, \ldots , b_n\}$ on $M$
determines on $\IR \otimes_{\IZ} M$ a Hilbert space structure by requiring that $\{1
\otimes b_1, 1 \otimes b_2, \ldots, 1 \otimes b_n\}$ is an orthonormal basis. Obviously
this Hilbert space structure depends only on the equivalence class $[B]$ of $B$.

We call a $\IZ$-chain complex $C_*$ \emph{finite based free} if it is finite free and each
$C_n$ comes with an equivalence class $[B_n]$ of $\IZ$-bases. Then $\IR \otimes{\IZ}C_*$
inherits the structure of a finite Hilbert chain complex.

\begin{lemma} \label{lem:homotopy_invariance_for_Z-chain_homotopy_equivalences}
Let $C_*$ be a finite  based free contractible $\IZ$-chain complex.
Then 
\[
\rho(\IR \otimes_{\IZ} C_*) = 0.
\]
\end{lemma}
\begin{proof}
  We use induction   over the length of $C_*$ which is the supremum $\{m-n \mid C_m \not= 0, C_n \not= 0\}$.
  The induction step, when the length is smaller than zero, is trivial since then $C_*$ is trivial. The
  induction step is done as follows.  Let $n$ be the smallest integer with $C_n  \not= 0$. 
  Then $c_{n+1} \colon C_{n+1} \to C_n$ is surjective. We can choose a map of
  $\IZ$-modules $s_n \colon C_n \to C_{n+1}$ with $c_{n+1} \circ s_n = \id_{C_n}$. Then
  the cokernel $\coker(s_n)$ is a finitely generated free and we can equip it with some
  equivalence class of $\IZ$-basis.  Let $\pr \colon C_{n+1} \to \coker(s_n)$ be the projection.
We obtain a short exact sequence of finite free $\IZ$-chain
  complexes by the following diagram
\[\xymatrix@!C=5em{\cdots \ar[r]
& 
0 \ar[r] \ar[d]
& 
0 \ar[r] \ar[d]
&
C_n \ar[r]^{\id} \ar[d]^{s_n}
&
C_n \ar[d]^{\id} 
\\
\cdots \ar[r]^{c_{n+4}} 
& 
C_{n+3} \ar[r]^{c_{n+3}} \ar[d]^{\id}
& 
C_{n+2} \ar[r]^{c_{n+2}} \ar[d]^{\id}
&
C_{n+1} \ar[r]^{c_{n+1}} \ar[d]^{\pr}
&
C_n \ar[d]
\\
\cdots \ar[r]^{c_{n+4}} 
& 
C_{n+3} \ar[r]^{c_{n+3}} 
& 
C_{n+2} \ar[r]^-{\pr \circ c_{n+2}} 
&
\coker(s_n) \ar[r]
&
0
}
\]
If we apply $\IR \otimes_{\IZ} -$ to the chain complex represented by the upper row, we
obtain a finite Hilbert chain complex with trivial torsion. The same is true by the induction hypothesis 
for the lower row since its length is smaller then the length of $C_*$.  Hence the claim follows
from Lemma~\ref{lem:additivity_of-torsion_fin_dim} if we can show the same for the
$2$-dimensional chain complex $E_*$ given in dimensions $0,1,2$ by $0 \to C_n
\xrightarrow{c_n} C_{n+1} \xrightarrow{\pr} \coker(s_n) \to 0$. Let $E_*'$ be the
$2$-dimensional chain complex $E_*$ given in dimensions $0,1,2$ by 
$0 \to C_n \to C_n \oplus \coker(s_n) \to \coker(s_n) \to 0$ where the
differentials are the obvious inclusion and projection and the $\IZ$-bases in dimension 1 is the
direct sum of the basis for $C_n$ and $\coker(s_n)$. Obviously we have $\rho(\IR
\otimes_{\IZ} E_*') = 0$. There is a $\IZ$-chain isomorphism $f_* \colon E_* \to E_*$ such
that $f_0$ and $f_2$ are the identity. We conclude from
Lemma~\ref{lem:isos_of_chain_complexes-fin_dim} that 
$\rho(\IR \otimes_{\IZ} E_*) = - \ln\bigl({\det}^{\perp}(\id_{\IR} \otimes_{\IZ} f_1)\bigr)$.  
Since $f_1$ is an isomorphism, ${\det}^{\perp}(\id_{\IR} \otimes_{\IZ} f_1)$ is the 
absolute value of the classical determinant of $\id_{\IR} \otimes_{\IZ} f_1$, which is the 
classical determinant of $f_1$ over $\IZ$ and hence $\pm 1$. This finishes the proof of
Lemma~\ref{lem:homotopy_invariance_for_Z-chain_homotopy_equivalences}.
\end{proof}

The term $\sum_{n \in \IZ} (-1)^n \cdot \ln\bigl({\det}^{\perp}(H_n(f_*))\bigr)$ 
appearing in Lemma~\ref{lem:homotopy_invariant_of_torsion_fin_dim}
causes some problems concerning homotopy
invariance as the following example shows:

\begin{example}[Subdivision for {$[0,1]$}]
\label{exa:subdivision_for_I}
Consider $I = [0,1]$. We specify a $CW$-structure on $I$ by defining the set of $0$-cells
by $\{0,1/n, 2 /n , \ldots, (n-1)/n, 1\}$ and the set of closed $1$-cells by $\{[0,1/n],
[1/n,2/n], \ldots [(n-1)/n,1]\}$ for each integer $n \ge 1$.  Denote the corresponding
$CW$-complex by $I[n]$.  The cellular $\IZ$-chain complex $C_*(I[n])$ is $1$-dimensional
and its first differential $c[n]_1 \colon \IZ^n \to \IZ^{n+1}$ is given by
\[
c[n]_1\bigl((k_1, k_2, \ldots, k_n)\bigr) =(-k_1, -k_2 + k_1, -k_3 + k_2, \ldots,  -k_n + k_{n-1}, k_n).
\]
The kernel of $c[n]_1$ is trivial and its image is the kernel
of the augmentation homomorphism $\epsilon[n] \colon \IZ^{n+1} \to \IZ, \;
(k_1, k_2, \ldots, k_{n+1}) \mapsto \sum_{i=1}^{n+1} k_i$.
In particular $H_1(C_*(I[n])) = 0$ and we get a $\IZ$-isomorphism 
\[
\overline{\epsilon[n]} \colon H_0(C_*(I[n])) \xrightarrow{\cong} \IZ
\]
induced by $\epsilon[n]$.
The Laplace operator $\Delta[n]_1 \colon \IR^n \to\IR^n$ in degree $1$ is given by the matrix
\[
A[n] =\begin{pmatrix}  
2 & -1   & 0 & 0 & \ldots & 0 & 0 & 0
\\
-1 & 2 & -1  & 0 & \ldots & 0 & 0 & 0
\\
0 & -1 & 2 & -1  & \ldots & 0 & 0 & 0
\\
0 & 0 & -1 & 2 &    \ldots & 0 & 0 & 0
\\
\vdots & \vdots & \vdots & \vdots & \ddots & \vdots & \vdots & \vdots
\\
0 & 0 & 0 & 0 & \ldots & 2 & -1 & 0
\\
0 & 0 & 0 & 0 & \ldots & -1 & 2 & -1 
\\
0 & 0 & 0 & 0 & \ldots & 0 & -1 & 2 
\end{pmatrix}
\]
By developing along the first row we get for its classical determinant for $n \ge 4$
\[
\det(A[n]) = 2 \cdot \det(A[n-1]) - \det(A[n-2]).
\]
A direct computation shows $\det(A[1]) = 2$, $\det(A[2]) = 3$ and $\det(A[3]) = 4$. This implies
$\det(A[n])  = n+1$ for all $n \ge 1$. Hence we get from Lemma~\ref{lem:torsion_in_terms_of_Laplacian}
\begin{eqnarray}
\rho(C_*(I[n])) 
& = & 
- \frac{1}{2} \cdot (-1)^{-1} \cdot 1 \cdot \ln\bigl({\det}^{\perp}(\Delta[n]_1)\bigr) 
\label{rho(C(I(n))}
\\
& = &
\frac{\ln\bigl(|{\det}(\Delta[n]_1)|\bigr))}{2}
\nonumber
\\ 
& = & \frac{\ln(n+1)}{2}.
\nonumber
\end{eqnarray}
This shows that $\rho(C_*(I[n])$ depends on the $CW$-structure. 

We have the chain map
$f \colon I[1] \to I[n]$ given by 
\begin{align*}
&f_1 \colon \IZ  \to  \IZ^n, \quad k \mapsto (k,k, \ldots , k);
\\
& f_0 \colon \IZ^2  \to  \IZ^{n+1} \quad  (k_1,k_2) \mapsto (k_1, 0,0,0 \ldots, k_2).
\end{align*}
It induces an isomorphism in homology since $\epsilon[n] \circ C_0(f_*) = \epsilon[1]$ holds.
Hence it is a $\IZ$-chain homotopy equivalence.
We conclude from Lemma~\ref{lem:homotopy_invariance_for_Z-chain_homotopy_equivalences}
\begin{eqnarray}
\rho\bigl(\cone(\id_{\IR} \otimes_{\IZ} C_*(f))\bigr) & = & 0.
\label{rho(cone(f(n))_is_0}
\end{eqnarray}
The isomorphism $\overline{\epsilon[n]} \colon H_0(C_*(I[n]) \xrightarrow{\cong} \IZ$
induces an explicite isomorphism
\[
\alpha[n] \colon H_0(\IR \otimes_{\IZ} C_*(I[n])) \xrightarrow{\cong} \IR \otimes_{\IZ} H_0(C_*(I[n]))
\xrightarrow{\id_{\IR} \otimes_{\IZ} \overline{\epsilon[n]}} \IR \otimes_{\IZ} \IZ \xrightarrow{\cong} \IR.
\]
Recall that $H_0(\IR \otimes_{\IZ} C_*(I[n]))$ inherits a Hilbert space structure. Then
$\alpha$ becomes an isometric isomorphism of Hilbert spaces if we equip $\IR$ with
the Hilbert space structure for which $1 \in \IR$ has norm $(n+1)^{-1/2}$, since the
element $(1,1, \ldots ,1) \in \IR \otimes_{\IZ} C_0(I[n]) = \IR^{n+1}$ belongs to 
$\ker(\id \otimes_{\IZ} \epsilon[n])^{\perp}$, has norm $\sqrt{n+1}$ and its class in 
$H_0(\IR \otimes_{\IZ} C_*(I[n])$ is sent  to $(n+1)$ under $\alpha[n]$. Since 
$\alpha[n] \circ H_0(f_*) = \alpha[1]$, we conclude
\begin{eqnarray}
{\ln\bigl((\det}^{\perp}(H_0(f))\bigr) = - \frac{\ln(n+1)}{2}.
\label{ln(det_perp(H_0f(n))}
\end{eqnarray}
Notice that~\eqref{rho(C(I(n))}, \eqref{rho(cone(f(n))_is_0}, and~\eqref{ln(det_perp(H_0f(n))}
are compatible with Lemma~\ref{lem:homotopy_invariant_of_torsion_fin_dim}.
\end{example}

Of course it cannot be desirable that $\rho(C_*(I[n]))$ in the
Example~\ref{exa:subdivision_for_I} depends on the $CW$-structure. This dependency is only
due to the dependency of the Hilbert structure on the homology on the
$CW$-structure. Therefore we can get rid of the dependency by fixing a Hilbert space
structure on the homology and view this as an extra piece of data. 

\begin{definition}[Torsion for finite based free $\IZ$-chain complex with a given Hilbert
  structure on homology]
  \label{def:Torsion_for_finite_free_Z-chain_complex_with_a_given_Hilbert_structure_on_homology}
  Let $C_*$ be a finite based free $\IZ$-chain complex. A \emph{Hilbert space structure
    $\kappa$} on $H_*(\IR \otimes_{\IZ} C_*)$  is a choice of Hilbert space structure $\kappa_n$ 
on each vector space   $H_n(\IR \otimes_{\IZ} C_*)$.  We define
\begin{multline*}
\rho(C_*;\kappa) := \rho(\IR \otimes_{\IZ} C_*) + 
\\
\sum_{n \in \IZ} (-1)^n \cdot \ln\bigl({\det}^{\perp}\bigl(\id \colon H_n(\IR \otimes_{\IZ} C_*) 
\to (H_n(\IR \otimes_{\IZ} C_*),\kappa(C_*)_n)\bigr)\bigr),
\end{multline*}
where on the source of 
$\id \colon H_n(\IR \otimes_{\IZ} C_*) \to (H_n(\IR \otimes_{\IZ} C_*),\kappa(C_*)_n)$ 
we use the Hilbert space structure induced by the one on $\IR \otimes_{\IZ} C_*$.
\end{definition}
If we take $\kappa$ to be the Hilbert space structure induced by the one on $\IR \otimes_{\IZ} C_*$, 
then obviously $\rho(C_*;\kappa)$ agrees with $\rho(\IR \otimes_{\IZ}
C_*)$. The desired effect  is the following  version of  homotopy invariance.

\begin{lemma}\label{lem:homotopy_invariance_of_rho(C_ast,kappa)}
  Let $f_* \colon C_* \to D_*$ be a $\IZ$-chain homotopy equivalence of finite based free
  $\IZ$-chain complexes.  Let $\kappa(C_*)$ and $\kappa(D_*)$ be Hilbert space structures
  on $H_*(\IR \otimes_{\IZ} C_*)$ and $H_*(\IR \otimes_{\IZ} D_*)$. Then we get
\begin{multline*}
\rho(D_*,\kappa(D_*)) - \rho(C_*,\kappa(C_*))
\\
=
\sum_{n \in \IZ} (-1)^n \cdot \ln\bigl({\det}^{\perp}\bigl(H_n(\id_{\IR} \otimes_{\IZ} f_*) \colon 
(H_n(\IR \otimes_{\IZ} C_*),\kappa(C_*)_n)
\\
\to  (H_n(\IR \otimes_{\IZ}  D_*),\kappa(D_*)_n)\bigr)\bigr).
\end{multline*}
\end{lemma}
\begin{proof} We get from Lemma~\ref{lem:homotopy_invariant_of_torsion_fin_dim} and
Lemma~\ref{lem:homotopy_invariance_for_Z-chain_homotopy_equivalences}
\begin{multline*}
\rho(\IR \otimes_{\IZ} D_*)  - \rho(\IR \otimes_{\IZ}  C_*) 
\\
= \sum_{n \in \IZ} (-1)^n \cdot \ln\bigl({\det}^{\perp}\bigl(H_n(\IR \otimes_{\IZ}  f_*) \colon 
H_n(\IR \otimes_{\IZ}  C_*) \to  H_n(\IR \otimes_{\IZ}  D_*)\bigr)\bigr).
\end{multline*}
Hence it suffices to show for each $n \in \IZ$
\begin{multline*}
{\det}^{\perp} (\id \colon H_n(\IR \otimes_{\IZ} C_*)  \to (H_n(\IR \otimes_{\IZ} C_*),\kappa(C_*)_n)\bigr) 
\\
\cdot
{\det}^{\perp}\bigl(H_n(\id_{\IR} \otimes_{\IZ} f_*) \colon  (H_n(\IR \otimes_{\IZ} C_*),\kappa(C_*)_n) 
\to  (H_n(\IR \otimes_{\IZ}  D_*),\kappa(D_*)_n)\bigr)
\\ 
= 
{\det}^{\perp}\bigl(H_n(\id_{\IR} \otimes_{\IZ} f_*) \colon  H_n(\IR \otimes_{\IZ} C_*) 
\to  H_n(\IR \otimes_{\IZ}  D_*)\bigr)
\\
\cdot
{\det}^{\perp} \bigl(\id \colon H_n(\IR \otimes_{\IZ} D_*)  \to (H_n(\IR \otimes_{\IZ} D_*),\kappa(D_*)_n)\bigr).
\end{multline*}
This follows from 
Lemma~\ref{lem:main_properties_of_det_perp}~\eqref{lem:main_properties_of_det_perp:composition}.
\end{proof}

\begin{example}[Integral Hilbert structure]
\label{exa:Integral_Hilbert_structure}
Let $C_*$ be a finite based free $\IZ$-chain complex. Choose for each integer $n$ a
$\IZ$-basis $B_n$ for $H_n(C_*)/\tors(H_n(C_*))$. Then we get an induced Hilbert structure
$\kappa[B_*]$ on $H_n(\IR \otimes_{\IZ} C_*)$ as follows.  Obviously $B_n$ induces an
$\IR$-basis on $\IR \otimes_{\IZ} H_n(C_*)/\tors(H_n(C_*))$. There is a canonical isomorphism
\[
\IR \otimes_{\IZ} H_n(C_*)/\tors(H_n(C_*)) \xrightarrow{\cong} H_n(\IR \otimes_{\IZ} C_*)
\]
We equip the target with the Hilbert space structure $\kappa(B_n)$ for which it
becomes an isometric isomorphism.

Now consider a chain homotopy equivalence $f_* \colon C_* \to D_*$ of finite based free
chain complexes.  Suppose that we have chosen $\IZ$-basis $B_n$ on $H_*(C_*)$ and $B'_n$
on $H_n(D_*)$.  Notice that $H_n(f)$ induces an isomorphism of $\IZ$-modules 
\[
H_n(C_*)/ \tors(H_n(C_*)) \xrightarrow{\cong} H_n(D_*)/ \tors(H_n(D_*))
\]
The determinant of it with respect to the given integral bases is $\pm 1$. One easily checks that
this implies
\[
{\det}^{\perp}\bigl(H_n(\id_{\IR} \otimes_{\IZ}  f_*) \colon (H_n(\IR\otimes_{\IZ}  C_*),\kappa(B_n)) 
\to (H_n(\IR \otimes_{\IZ}  D_*),\kappa(B'_n)\bigr) = 1.
\]
Lemma~\ref{lem:homotopy_invariance_of_rho(C_ast,kappa)} implies
\[
\rho(C_*;\kappa(B_*)) = \rho(D_*;\kappa(B'_*)).
\]
Hence $\rho(C_*;\kappa(B_*))$ is independent of the choice of integral basis on $C_*$,
$D_*$, $H_n(C_*)$, and $H_n(D_*)$ and is a homotopy invariant of the underlying finite
free $\IZ$-chain complexes $C_*$ and $D_*$. This raises the question what it is?

We leave it to the reader to figure out 
\[
\rho(C_*;\kappa(B_*)) = \sum_{n \in \IZ} (-1)^n \cdot \ln\bigl(|\tors(H_n(C_*))|\bigr).
\]
The proof is straightforward after one has shown using the fact $\IZ$ is a principal ideal
domain that $C_*$ is homotopy equivalent to a direct sum of $\IZ$-chain complexes each of
which is concentrated in two consecutive dimensions and given there by $m \cdot \IZ \to
\IZ$ for some integer $m \in \IZ$.
\end{example}


\typeout{------------------------ Section 3:  The Hodge de Rham Theorem  --------------------------}

\section{The Hodge de Rham Theorem}
\label{sec:The_Hodge_deRhamTheorem}

Next we want to give a first classical relation between topology and analysis, the de Rham
Theorem and the Hodge-de Rham Theorem.


\subsection{The de Rham Theorem}
\label{subsec:The_deRham_Theorem}

Let $M$ be a (not necessarily compact) manifold (possibly with boundary).

The \emph{de Rham complex}
$(\Omega^*(M),d^*)$ is the real cochain complex whose $n$-th chain module is the real
vector space of smooth $n$-forms on $M$ and whose  $n$-differential is the standard
differential for $n$-forms. The \emph{de Rham cohomology} of $M$ is defined by
\begin{equation}
H^n_{\dR}(M) := \ker(d^n)/\im(d^{n-1}).
\label{deRham_cohomology}
\end{equation}
There is a $\IR$-chain  map, natural in $M$,
\[
A^*(M) \colon \Omega^*(M) \to C^*_{\sing,C^{\infty}}(M;\IR) 
\]
with the cochain complex of $M$ based on smooth singular simplices with coefficients in $\IR$ as
target. It sends an $n$-form $\omega \in \Omega^n(M)$ to the element 
$A^n(\omega) \in C^n_{\sing,C^{\infty}}(;\IR))$ which assigns to a smooth singular $n$-simplex 
$\sigma \colon \Delta_n \to M$  the real number $\int_{\Delta_n} \sigma^*\omega$.  
The Theorem of Stokes implies that this is a chain map.  There is a forgetful chain map 
\[
C^*_{\sing;C^{\infty}}(M;\IR) \to C^*_{\sing}(M;\IR)
\]
to the standard singular $\IR$-cochain complex, which is based on (continuous) singular
simplices with coefficients in $\IR$. Denote by $H^*_{\sing,C^{\infty}}(M;\IR)$ the smooth
singular cohomology of $M$ with coefficients in $\IR$ which is by definition the
cohomology of the $\IR$-cochain complex $C^*_{\sing;C^{\infty}}(M;\IR)$, and define analogously
$H^*_{\sing}(M;\IR)$.  A proof of the next theorem, at least in the case $\partial M = \emptyset$,
can be found  for instance 
in~\cite[Section~V.9.]{Bredon(1997a)},~\cite{deRham(1984)},~%
\cite[Theorem~1.5 on page~11 and Theorem~2.4 on page~20]{Dupont(1978)},~%
\cite[Section~15]{Lueck(2005algtop)},~\cite[Theorem~A.31 on page~413]{Massey(1991)}.

\begin{theorem}[De Rham Theorem]
\label{the:De_Rham_Theorem}
The chain map $A^*$ induces for a smooth manifold $M$ 
and $n \ge 0$ an isomorphism, natural in $M$,
\[
H^n(A^*(M)) \colon H^*_{\dR}(M) \xrightarrow{\cong} H^n_{\sing,C^{\infty}}(M;\IR).
\]
The forgetful chain map induces an isomorphism, natural in $M$,
\[
H^n_{\sing,C^{\infty}}(M;\IR) \xrightarrow{\cong} H^n_{\sing}(M;\IR).
\]
They are compatible with the multiplicative structures
given by the $\wedge$-product and the $\cup$-product.
\end{theorem}


\subsection{The Hodge-de Rham Theorem}
\label{subsec:The_Hodge_deRham_Theorem}

Now suppose that the smooth  manifold $M$ 
comes with a Riemannian metric and an orientation. Let $d$ be the
dimension of $M$. Denote by
\begin{eqnarray}
\ast^n \colon  \Omega^n(M) & \rightarrow & \Omega^{d-n}(M)
\label{Hodge_star-operator}
\end{eqnarray}
the \emph{Hodge star-operator} which is defined by the corresponding notion for oriented
finite-dimensional Hilbert spaces applied fiberwise.  It is uniquely characterized by the
property
\begin{eqnarray}
\int_M \omega \wedge \ast^n\eta
& = &
\int_M \langle \omega_x, \eta_x\rangle_{\Alt^n(T_xM)} \dvol,
\label{characterization_of_the_Hodge_star_operator}
\end{eqnarray}
where $\omega$ and $\eta$ are $n$-forms, $\omega$ has compact support,
and $\langle \omega_x,
\eta_x\rangle_{\Alt^n(T_xM)}$ is the inner product on $\Alt^n(T_xM)$ which is induced by
the inner product on $T_xM$ given by the Riemannian metric.

Define the \emph{adjoint of the exterior differential}
\begin{eqnarray}
\delta^n
= (-1)^{dn + d + 1} \cdot *^{d-n+1} \circ \,  d^{d-n} \circ *^n\colon  \Omega^n(M)
& \rightarrow &
\Omega^{n-1}(M).
\label{adjoint_of_the_exterior_differential}
\end{eqnarray}
Notice that in the definition of $\delta^n$ the Hodge star-operator appears twice and the
definition is local. Hence we can define $\delta^n$ without using an orientation of $M$,
only the Riemannian metric is needed.  This is also true for the following definition.

\begin{definition}[Laplace operator]
\label{def:Laplace_operator}
Define the \emph{$n$-th Laplace operator} on the  Riemannian manifold $M$ 
\[
\Delta_n
 = d^{n-1}\circ \delta^n + \delta^{n+1} \circ d^n\colon  \Omega^n(M) \to \Omega^n(M). 
\]
\end{definition}

Let $\Omega_c^n(M) \subset \Omega^n(M)$
be the \emph{space of smooth $p$-forms with compact support}.
There is the following inner product on it
\begin{eqnarray}
\langle \omega, \eta\rangle_{L^2}
& := &
\int_M \omega \wedge *^n\eta  = 
\int_M \langle \omega_x , \eta_x\rangle_{\Alt^n(T_xM)} \dvol.
\label{inner_L2-product_on_space_of_smooth_forms_with_compact_support}
\end{eqnarray}

Recall that a Riemannian manifold $M$ is \emph{complete} if each path component of $M$
equipped with the metric induced by the Riemannian metric is a complete metric space.  By
the Hopf-Rinow Theorem the following statements are equivalent provided that $M$ has no
boundary: (1) $M$ is complete, (2) the exponential map is defined for any point $x \in M$
everywhere on $T_xM$, (3) any geodesic of $M$ can be extended to a geodesic defined on
$\IR$, see~\cite[page~94 and~95]{Gallot-Hulin-Lafontaine(1987)}.  Completeness enters in
a crucial way, namely, it will allow us to integrate by parts~\cite{Gaffney(1954)}.

\begin{lemma} \label{lem:integration_by_parts_on_complete_manifolds}
Let $M$ be a complete Riemannian manifold.
Let $\omega \in \Omega^n(M)$ and $\eta \in \Omega^{n+1}(M)$ be smooth forms
such that $\omega$, $d^n\omega$, $\eta$ and $\delta^{n+1}\eta$
are square-integrable.
Then
\[
\langle d^n\omega, \eta\rangle_{L^2}  - 
\langle \omega, \delta^{n+1}\eta\rangle_{L^2} =
\int_{\partial M} (\omega \wedge \ast^{n+1} \eta)|_{\partial M}.
\]
\end{lemma}
\begin{proof}
Completeness ensures the existence of a sequence $f_n\colon M \to [0,1]$
of smooth functions with compact support such that
$M$ is the union of the compact sets $\{x \in M \mid f_n(x) = 1\}$ and
$||df_n||_{\infty} := \sup\{||(df_n)_x||_x \mid x\in M\}  < \frac{1}{n}$
holds. With the help of the sequence $(f_n)_{n \ge 1}$
one can reduce the claim to the easy case, where $\omega$ and $\eta$ have
compact support.  
\end{proof}

From now on suppose that the boundary of $M$ is empty.
Then $d^n$ and $\delta^n$ are formally adjoint in the sense
that we have for $\omega \in \Omega^n(M)$ and $\eta \in \Omega^{n+1}(M)$
such that $\omega$, $d^n\omega$, $\eta$ and $\delta^{n+1}\eta$ are square-integrable.
\begin{eqnarray}
\langle d^n(\omega),\eta\rangle_{L^2}
& = &
\langle \omega, \delta^{n+1}(\eta)\rangle_{L^2}.
\label{d_and_delta_are_adjoint_on_smooth_forms_with_compact_support}
\end{eqnarray}
Let $L^2\Omega^n(M)$
be the Hilbert space completion of $\Omega_c^n(M)$.
Define the \emph{space of $L^2$-integrable harmonic smooth $n$-forms}
\begin{equation}
\calh^n_{(2)}(M)  := 
\{\omega \in \Omega^n(M) \mid \Delta_n(\omega) = 0,
\int_M \omega \wedge *\omega < \infty\}.
\label{space_of_L2-integrable_harmonic_smooth_n-forms}
\end{equation}

The following two results are the analytic versions of Lemma~\ref{lem:Hodge_decomposition}.

\begin{theorem}[Hodge-de Rham Decomposition]
\label{the:Hodge-de_Rham_Decomposition}
Let $M$ be a complete Riemannian manifold without boundary. Then we
obtain an orthogonal decomposition, the so called
\emph{Hodge-de Rham decomposition}
\begin{eqnarray*}
L^2\Omega^n(M)
& = &
\calh_{(2)}^n(M) \oplus \clos\bigl(d^{n-1}(\Omega_c^{n-1}(M))\bigr) \oplus
\clos\bigl(\delta^{n+1}(\Omega_c^{n+1}(M))\bigr).
\end{eqnarray*}
\end{theorem}

For us the following result will be of importance.
Put
\begin{equation}
\calh^n(M)  := 
\{\omega \in \Omega^n(M) \mid \Delta_n(\omega) = 0\}.
\label{space_of_harmonic_smooth_n-forms}
\end{equation}
This is the same as $\calh^n_{(2)}(M)$ introduced in~\eqref{space_of_L2-integrable_harmonic_smooth_n-forms}
if $M$ is compact.

\begin{theorem}[Hodge-de Rham Theorem]
\label{the:Hodge_de_Rham_Theorem}
Let $M$ be a closed smooth manifold. Then the canoncial map
\[
\calh^n(M) \xrightarrow{\cong} H^n_{\dR}(M)
\]
is an isomorphism.
\end{theorem}
\begin{proof} See for instance \cite[Lemma~1.5.3]{Gilkey(1994)}, or~\cite[(4.2)]{Roe(1988a)}.
\end{proof}

The following remarks are the analytic versions of 
Remark~\ref{rem:Homotopy_invariance_of_dim(ker(Delta_n)_combinatorial} and 
Remark~\ref{rem:heat_operator_hilbert}

\begin{remark}[Homotopy invariance of $\dim(\calh^n(M))$]
  \label{rem:Homotopy_:invariance_of_dim(ker(Delta_n)}
  Theorem~\ref{the:Hodge_de_Rham_Theorem} implies that $\dim(\ker(\Delta_n))$
  depends only on the homotopy type of $M$ and is in particular
  independent of the Riemannian metric of $M$.  Of course the spectrum
  of the Laplace operator $\Delta_n$ does depend on the Riemannian metric, 
  but a part of it, namely, the multiplicity of the
  eigenvalue $0$, which is just $\dim(\calh^n(M))$, depends only
  on the homotopy type of $M$. 
\end{remark}

\begin{remark}[Heat kernel]\label{rem:heat_operator_analytic}
  To the analytic Laplace operator $\Delta_n \colon \Omega^n M \to \Omega^n M$ one can assign its
  \emph{heat operator} $e^{-t\Delta_n} \colon \Omega^n M \to \Omega^n M$ using functional
  calculus. Roughly speaking, each eigenvalue $\lambda$ of $\Delta_n$ transforms to the
  eigenvalue $e^{-t\lambda}$.  This operator runs out to be given by a kernel, the so
  called \emph{heat kernel} $e^{-t \Delta_n}(x,y)$.  Recall that $e^{-t \Delta_n}(x,y)$ is
  an element in $\hom_{\IR}(\Alt^n(T_xM),\Alt^n(T_yM))$ for $x,y$ in $M$ and we get for
  $\omega \in \Omega^n(M)$
  \[
  e^{-t\Delta_n}(\omega)_x = \int_M e^{-t\Delta_n}(x,y)(\omega_y) \dvol.
  \]
  For each $x \in M$ we obtain an endomorphism $e^{-t \Delta_n}(x,x)$ of a
  finite-dimensional real vector space and we have the real number
  $\tr\bigl(e^{-t\Delta_n}(x,x)\bigr)$. Then we get, see~\cite[1.6.52 on page
  56]{Gilkey(1994)}
  \[
  b_n(M) = \lim_{t \to \infty} \int_M \tr\bigl(e^{-t\Delta_n}(x,x)\bigr) \dvol.
  \]
\end{remark}


\typeout{--------------- Section 4: Reidemeister torsion for closed Riemannian manifolds  ----------------}

\section{Topological torsion for closed Riemannian manifolds}
\label{sec:Topological_torsion_for_closed_Riemannian_manifolds}

In the section we introduce and investigate the notion of the topological torsion for a closed Riemannian
manifold.


\subsection{The definition of topological torsion for closed Riemannian manifolds}
\label{subsec:The_definition_of_Topological_torsion_for_closed_Riemannian_manifolds}

Let $M$ be a closed Riemannian manifold. 
The Riemannian metric induces an inner product on $\Omega^n(M)$, 
see~\eqref{inner_L2-product_on_space_of_smooth_forms_with_compact_support},
and hence a Hilbert space structure on the finite-dimensional real vector space
$\calh^n(M)$. Equip $H^n_{\sing}(M;\IR)$ with the Hilbert space structure $\kappa^n_{\harm}(M)$
for which the composite of the isomorphisms (or their inverses) of Theorem~\ref{the:De_Rham_Theorem}
and Theorem~\ref{the:Hodge_de_Rham_Theorem}
\[
H^n_{\sing}(M;\IR) \xrightarrow{\cong} H^n_{\sing;C^{\infty}}(M;\IR)
\xrightarrow{\cong} H^n_{\dR}(M) \xrightarrow{\cong} \calh^n(M)
\]
becomes an isometry. There is a preferred isomorphism
\[
\hom_{\IR}(H_n^{\sing}(M;\IR),\IR) \xrightarrow{\cong} H^n_{\sing}(M;\IR).
\]
Equip $H_n^{\sing}(M;\IR)$ with the Hilbert space structure
$\kappa_n^{\harm}(M)$, such that for the induced Hilbert space structure on the dual vector space
$\hom_{\IR}(H_n^{\sing}(M;\IR),\IR)$ and the Hilbert space structure $\kappa^n_{\harm}(M)$ on 
$H^n_{\sing}(M;\IR)$ introduced above this isomorphisms becomes an isometry.

Fix a finite $CW$-complex $X$ and a homotopy equivalence
$f \colon X \to M$, for instance, a smooth triangulation $t \colon K \to M$,
i.e., a finite simplicial complex $K$ together with a homeomorphism $t \colon K \to M$ such
that the restriction of $t$ to a simplex is a smooth immersion,
see~\cite{Munkres(1961),Whitehead(1940_complexes)}.  
Recall that there is a natural isomorphism between singular and cellular homology
\[
u_n(X;\IR) \colon H_n(X;\IR) := H_n(\IR \otimes_{\IZ} C_*(X)) \xrightarrow{\cong} H_n^{\sing}(X;\IR).
\]
We equip $H_n(X;\IR) := H_n(\IR \otimes_{\IZ} C_*(X))$ with the Hilbert space structure $\kappa_n(f)$ for which the 
preferred isomorphism
\[
H_n(X;\IR)  \xrightarrow{u_n(X;\IR)} H_n^{\sing}(X;\IR) \xrightarrow{H_n^{\sing}(f;\IR)} H_n^{\sing}(M;\IR)
\]
is isometric if we equip the target with the Hilbert space structure $\kappa^{\harm}_n(M)$ introduced above.

The cellular $\IZ$-chain complex $C_*(X)$ inherits from the $CW$-structure a preferred equivalence of $\IZ$-basis.
So we can consider
\begin{equation*}
\rho(C_*(X);\kappa_*^{\harm}(f)) \in \IR
\end{equation*}
as introduced in 
Definition~\ref{def:Torsion_for_finite_free_Z-chain_complex_with_a_given_Hilbert_structure_on_homology}.
Consider another finite $CW$-complex $X'$ and a homotopy equivalence
$f' \colon X' \to M$. Choose a cellular homotopy equivalence $g \colon X \to X'$ such that
$f' \circ g$ is homotopic to $f$. Then $C_*(g) \colon C_*(X) \to C_*(X')$ is a 
$\IZ$-chain homotopy equivalence of finite based free $\IZ$-chain complexes such that 
$H_n(g;\IR) \colon (H_n(X;\IR), \kappa^{\harm}_n(f)) \to (H_n(X';\IR), \kappa^{\harm}_n(f')))$
is an isometric isomorphism for all  $n \ge 0$.  We conclude from
Lemma~\ref{lem:homotopy_invariance_of_rho(C_ast,kappa)} 
\[
\rho(C_*(X);\kappa_*^{\harm}(f))  = \rho(C_*(X');\kappa_*^{\harm}(f')).
\]
Hence the following definition makes sense.

\begin{definition}[Topological torsion of a closed Riemannian manifold]
\label{def:Topological_Torsion_of_a_closed_Riemannian_manifold}
Let $M$ be a closed Riemannian manifold. Define its \emph{topological torsion}
\[
\rho_{\topo}(M) := \rho(C_*(X),\kappa^{\harm}_*(f))
\]
for any choice of finite $CW$-complex $X$ and homotopy equivalence $f \colon X \to M$.
\end{definition}


\subsection{Topological torsion of rational homology spheres}
\label{subsec:Topological_torsion_of_rational_homology_spheres}

Let $M$ be a closed oriented Riemannian manifold which is a rational homology sphere,
i.e.,  $H_n(M;\IQ) \cong H_n(S^d;\IQ)$ for $d = \dim(M)$ and $n \ge 0$.
We want to show
\begin{equation}
\rho^{\topo}(M) 
 = 
\frac{1- (-1)^d}{2} \cdot \ln(\vol(M)) 
+ \sum_{n \ge 0} (-1)^n\cdot \ln\bigl(\bigl|\tors(H_n(M;\IZ))\bigr|\bigr).
\label{rho_topo(rational_homology_sphere)}
\end{equation}
Choose a finite $CW$-complex $X$ and a homotopy equivalence $f \colon X \to M$.
If we equip $H_*(\IR \otimes_{\IZ} C_*(X))$ with the integral Hilbert space structure $\kappa^{\IZ}_*$
as explained in Example~\ref{exa:Integral_Hilbert_structure}, we get
from Example~\ref{exa:Integral_Hilbert_structure}. 
\[
\rho(\IR \otimes_{\IZ} C_*(X);\kappa^{\IZ}_*) 
= \sum_{n \ge 0} (-1)^n\cdot \ln\bigl(\bigl|\tors(H_n(M;\IZ))\bigr|\bigr).
\]
Hence we get
\begin{eqnarray*}
\rho^{\topo}(M) 
& = & 
\rho(X;\kappa^{\harm}_*)
\\
 & = & 
\rho(X;\kappa^{\harm}_*) - \rho(X;\kappa^{\IZ}_*) + \rho(X;\kappa^{\IZ}_*) + 
\\ 
& = & 
\rho(X;\kappa^{\harm}_*) - \rho(X;\kappa^{\IZ}_*) 
+ \sum_{n \ge 0} (-1)^n\cdot \ln\bigl(\bigl|\tors(H_n(M;\IZ))\bigr|\bigr).
\end{eqnarray*}
 Lemma~\ref{lem:homotopy_invariance_of_rho(C_ast,kappa)} implies
\begin{eqnarray*}
\lefteqn{\rho(X;\kappa^{\harm}_*) - \rho(X;\kappa^{\IZ}_*)}
& & 
\\
& = &
\ln\bigl({\det}^{\perp}\bigl(\id \colon H_0(\IR\otimes_{\IZ} C_*(X)),\kappa^{\IZ}_0(X)) \to 
H_0(\IR \otimes_{\IZ} C_*(X)),\kappa_0^{\harm}(X))\bigr)
\\
& & 
+(-1)^d \cdot
\ln\bigl({\det}^{\perp}\bigl(\id \colon H_d(\IR\otimes_{\IZ} C_*(X)),\kappa^{\IZ}_d(X)) \to 
\\
& & \hspace{60mm}
H_d(\IR \otimes_{\IZ} C_*(X)),\kappa^{\harm}_d(X))\bigr).
\end{eqnarray*}

Let $1 \in H_0^{\sing}(M;\IZ)$ and $[M] \in H_d^{\sing}(M;\IZ)$ be the obvious generators of
the infinite cyclic groups $H_0^{\sing}(M;\IZ)$ and $H_d^{\sing}(M);\IZ)$. They determine
elements in the $1$-dimensional vector spaces $H^0_{\sing}(M;\IR) =
\hom_{\IZ}(H_0(M;\IZ),\IR)$ and $H^d_{\sing}(M;\IR) = \hom_{\IZ}(H_d(M;\IZ),\IR)$. Their 
image under the composite
\[
H^n_{\sing}(M;\IR) \xrightarrow{\cong} H^n_{\sing;C^{\infty}}(M;\IR)
\xrightarrow{\cong} H^n_{\dR}(M) \xrightarrow{\cong} \calh^n(M)
\]
is the constant function $c_1 \colon M \to \IR$ with value $1$ and $\frac{\dvol}{\vol(M)}$
for $\dvol$ the volume form $M$ for $n = 0,d$.  The norm of $c_1$ and
$\frac{\dvol}{\vol(M)}$ with respect to norm coming 
from~\eqref{inner_L2-product_on_space_of_smooth_forms_with_compact_support} is
\[
||c_1||_{L^2} = \sqrt{\int_M c_1 \wedge \ast^d(c_1)} = \sqrt{\int_M \dvol} = \sqrt{\vol(M)},
\]
and
\[
\left|\left|\frac{\dvol}{\vol(M)}\right|\right|_{L^2} 
= 
\sqrt{\int_M \frac{\dvol}{\vol(M)} \wedge \ast^d\left(\frac{\dvol}{\vol(M)}\right)} 
=  
\sqrt{\int_M \frac{\dvol}{\vol(M)^2}} 
= 
\sqrt{\frac{1}{\vol(M)}}.
\]
This implies 
\begin{multline*}
\ln\bigl({\det}^{\perp}\bigl(\id \colon H_0(\IR\otimes_{\IZ} C_*(X)),\kappa^{\IZ}_0(X)) \to 
H_0(\IR \otimes_{\IZ} C_*(X)),\kappa_n^{\harm}(X))\bigr) 
\\
= \frac{\ln(\vol(M))}{2}
\end{multline*}
and
\begin{multline*}
\ln\bigl({\det}^{\perp}\bigl(\id \colon H_d(\IR\otimes_{\IZ} C_*(X)),\kappa^{\IZ}_d(X)) \to 
H_d(\IR \otimes_{\IZ} C_*(X)),\kappa_d^{\harm}(X))\bigr) 
\\
= \frac{-\ln(\vol(M))}{2}
\end{multline*}
Now~\eqref{rho_topo(rational_homology_sphere)} follows.


\subsection{Further properties of the topological torsion}
\label{subsec:Further_properties_of_the_topological_torsion}

Lemma~\ref{lem:homotopy_invariance_of_rho(C_ast,kappa)}  implies 

\begin{lemma} \label{lem:homotopy_invariance_of_the_topological_Reidemeister_torsion}
Let $f \colon M \to N$ be a homotopy equivalence of closed Riemannian manifolds. Then
\begin{multline*}
\rho_{\topo}(N) - \rho_{\topo}(M) =  \sum_{n \ge 0} (-1)^n \cdot  
{\det}^{\perp}\bigl(H_n^{\sing}(f;\IR) \colon H_n^{\sing}(M,\kappa^{\harm}_n(M)) 
\\
\to H_n(N;\IR),\kappa_n^{\harm}(N))\big).
\end{multline*}
\end{lemma}

\begin{remark}[Twisting with finite-dimensional orthogonal representations]
  \label{rem:Twisting_with_finite-dimensional_orthogonal_representations}
  In general the topological torsion does depend on the Riemannian metric, see
  Lemma~\ref{lem:homotopy_invariance_of_the_topological_Reidemeister_torsion}.
  Nevertheless the name topological torsion is justified since this dependency is well
  understood and depends only on $H_n(M;\IR)$.

  Notice that at least $H_0(M;\IR)$ cannot be trivial for a smooth manifold. However, there
  are prominent cases, where one can specify a orthogonal finite-dimen\-sio\-nal
  representation $V$ of $\pi_1(M)$ for a closed Riemannian manifold $M$ such that the
  $V$-twisted singular homology $H_n^{\pi_1(M)}(M;V)$ vanishes for all $n \ge 0$. One can
  also define a $V$-twisted topological torsion $\rho(M;V)$. If $H_n^{\pi_1(M)}(M;V)$
  vanishes for all $n \ge 0$, then $\rho(M;V)$ does not depend on the Riemannian metric at
  all, and  only on the simple homotopy type of $M$.
\end{remark}

\begin{remark}[Poincar\'e duality] \label{rem:Poincare_duality}
A direct computation using Poincar\' e duality and the Universal Coefficient Theorem
show that in the situation of 
Subsection~\ref{subsec:Topological_torsion_of_rational_homology_spheres}
the topological torsion vanishes if the dimension of $M$ is even. This is true in general.
Namely, if $M$ is a closed Riemannian manifold of even dimension, then 
$\rho_{\topo}(M) = 0$. 
\end{remark}

\begin{remark}[Product formula]
\label{rem:product_formula} Let $M$ be closed Riemannian manifolds. Then
\[
\rho_{\topo}(M \times N) = \chi(M) \cdot \rho_{\topo}(N) + \chi(N) \cdot \rho_{\topo}(M).
\]
One can more generally investigate the behavior of the topological torsion under fiber bundles,
see~\cite{Lueck-Schick-Thielmann(1998)}.
\end{remark}

\begin{remark}[Compact manifolds with boundary and glueing formula]
\label{rem:Compact_manifolds_with_boundary_and_glueing_formula}
The topological torsion is also defined for compact Riemannian manifolds with boundary.
One has to put the right boundary conditions on the space of harmonic forms so that
Theorem~\ref{the:Hodge_de_Rham_Theorem} remains true. 

Consider compact Riemannian manifolds $M$ and $N$ together with a diffeomorphism $f
\colon \partial M \xrightarrow{\cong} \partial N$.  Equip $M$, $N$, $\partial N$, and $M
\cup_f N$ with Riemannian metrics. Then one obtains the glueing formula
\[
\rho_{\topo}(M \cup_f N) = \rho_{\topo}(M) + \rho_{\topo}(N) - \rho_{\topo}(\partial M) + \rho(LHS_*),
\]
where $LHS_*$ is the Hilbert chain complex given by the long exact homology sequence
\begin{multline*}
\ldots \to H_n^{\sing}(\partial M;\IR) \to H_n^{\sing}(M;\IR) \oplus H_n^{\sing}(N;\IR) 
\\
\to 
H_n^{\sing}(M \cup_f N;\IR) \to H_{n-1}^{\sing} (\partial M; \IR) \to \ldots
\end{multline*}
for which each homology group is equipped with the harmonic Hilbert space structure $\kappa^{\harm}$.
This follows from Lemma~\ref{lem:additivity_of-torsion_fin_dim}
and Lemma~\ref{lem:homotopy_invariance_of_rho(C_ast,kappa)}.
\end{remark}


\typeout{--------------- Section 5: Analytic torsion for closed Riemannian manifolds  ----------------}

\section{Analytic torsion for closed Riemannian manifolds}
\label{sec:Analytic_torsion_for_closed_Riemannian_manifolds}

Recall that we showed that the Betti number of a finite
$CW$-complex $X$ is the dimension of the kernel of the combinatorial Laplace operator
$\Delta_n \colon \IR \otimes_{\IZ} C_n(X) \to \IR \otimes_{\IZ} C_n(X)$.  This triggered
the question whether the Betti number $b_n(M)$ of a closed Riemannian manifold is the dimension of
the kernel of the analytic Laplace operator $\Delta_n \colon \Omega^n(M) \to \Omega(M)$.
We saw  that the answer is positive, see Theorem~\ref{the:Hodge_de_Rham_Theorem}.

Next we want to apply the same line of thought to  torsion. We know how to express the topological
torsion in terms of the combinatorial Laplace operator by
Lemma~\ref{lem:torsion_in_terms_of_Laplacian}, namely for a finite
$CW$-complex $X$ and a homotopy equivalence $X \to M$ we get for the combinatorial Laplace
operator $\Delta_n \colon C_n(X) \to C_n(X)$ the formula
\begin{multline*}
\rho^{\topo}(M) := -\frac{1}{2} \cdot \sum_{n \ge 0} (-1)^n \cdot n \cdot \ln\bigl({\det}^{\perp}(\Delta)_n\bigr)
\\
+ \sum_{n \ge 0} (-1)^n \cdot \ln\bigl({\det}^{\perp}\bigl(\id \colon H_n(\IR \otimes_{\IC} C_*) \to 
(H_n(\IR \otimes_{\IC} C_*),\kappa^{\harm}_n)\bigr)\bigr).
\end{multline*}
One can hope that the rather complicated correction term given by the sum of terms
involving $\kappa^{\harm}_*$ is not necessary in the analytic setting, since the analytic
Laplace operator $\Delta_n \colon \Omega^n(M) \to\Omega^n(M)$  
is closely related  to harmonic forms. This suggests
to try to make sense of the following expression involving the  analytic Laplace operator

\[
\rho_{\an}(M) := -\frac{1}{2} \cdot \sum_{n \ge 0} (-1)^n \cdot n \cdot \ln\bigl({\det}^{\perp}(\Delta_n)\bigr).
\]
The problem is that the analytic Laplace operator  $\Delta_n$
acts on infinite-dimensional vector spaces and therefore  the expression
${\det}^{\perp}(\Delta_n)$ is a priori not defined. To give it nevertheless a meaning, one
has to take a closer look on the spectrum of the analytic Laplace operator $\Delta_n$ for
a closed Riemannian manifold.


\subsection{The spectrum of the Laplace operator on closed Riemannian manifolds}
\label{subsec:The_spectrum_of_the_Laplacian_on_closed_Riemannian_manifolds}

Let $M$ be a closed Riemannian manifold. Next we record some basic facts about the
spectrum of the analytic Laplace operator $\Delta_n \colon \Omega^n(M) \to \Omega^n(M)$. Denote by 
$E_{\lambda}(\Delta_n) = \{\omega \in \Omega^n(M) \mid \Delta_n(\omega) = \lambda \cdot\omega\}$ 
the eigenspace of $\Delta_n$ for $\lambda \in \IC$. We call $\lambda$ an
eigenvalue of $\Delta_n$ if $E_{\lambda}(\Delta_n) \not= 0$. It turns out that each
eigenvalue $\lambda$ of $\Delta_n$ is a real number satisfying $\lambda \ge 0$.  Notice
that $E_{\lambda}(\Delta_n)$ and $E_{\mu}(\Delta_n)$ are orthogonal in 
$L^2\Omega^n(M)$ for $\lambda \not= \mu$ since we get
from~\eqref{d_and_delta_are_adjoint_on_smooth_forms_with_compact_support} for 
$\nu_0, \nu_1 \in \Omega^n(M)$
\begin{equation}
\langle \Delta_n(\nu_0),\nu_1\rangle_{L^2} = \langle \nu_0,\Delta_n(\nu_1) \rangle_{L^2},
\label{Delta_n_selfadjoint}
\end{equation}
and hence we get for $\omega \in E_{\lambda}(\Delta_n)$ and $\eta \in E_{\mu}(\Delta_n)$ 
\[
\lambda \cdot \langle \omega, \eta\rangle_{L^2}
=
\langle \lambda \cdot \omega, \eta\rangle_{L^2}
\\
=
\langle \Delta_n(\omega), \eta\rangle_{L^2} 
\\
=
\langle \omega, \Delta_n(\eta)\rangle_{L^2}
\\
 =  
\langle \omega, \mu \cdot \eta\rangle_{L^2}
\\
 = 
\mu \cdot \langle \omega, \eta\rangle_{L^2}.
\]
Moreover, we have the orthogonal decomposition
\[
\bigoplus_{\lambda \ge 0} E_{\lambda}(\Delta_n) = L^2 \Omega^n(M).
\]
We define the \emph{$n$-th-Zeta-function} for $s \in \IC$
\begin{eqnarray}
\zeta_n(s) = \sum_{\lambda > 0} \dim_{\IR}(E_{\lambda}(\Delta_n)) \cdot \lambda^{-s},
\label{zeta_n}
\end{eqnarray}
where $\lambda$ runs through all eigenvalues of $\Delta_n$ with $\lambda > 0$.
Of course it is a priori not clear whether this sums converges. However, the following 
result holds, see for instance~\cite[Section~1.12]{Gilkey(1994)}.

\begin{lemma} \label{lem:meroextension} The Zeta-function $\zeta_n$
  converges absolutely for $s \in S = \{s \in \IC \mid \Real(s) >
  \dim(M)/2\}$ and defines a holomorphic function on $S$. It has  a
  meromorphic extension to $\IC$ which is analytic in zero and whose
  derivative at zero $\left.\frac{d}{ds}\right|_{s = 0} \zeta_n(s)$ lies in $\IR$.
\end{lemma}


\subsection{The definition of analytic torsion for closed Riemannian manifolds}
\label{subsec:The_definition_of_analytic_torsion_for_closed_Riemannian_manifolds}

In view of Lemma~\ref{lem:meroextension} the
following definition make sense.  It is due to Ray-Singer~\cite{Ray-Singer(1971)} and
motivated by~\eqref{det_and_zeta_fin_dim} and
Lemma~\ref{lem:torsion_in_terms_of_Laplacian}.

\begin{definition}[Analytic torsion]
\label{def:analytic_torsion}
Let $M$ be a closed Riemannian manifold. Define its \emph{analytic torsion}
\[
\rho_{\an}(M) := \frac{1}{2} \cdot \sum_{n \ge 0} (-1)^n \cdot n \cdot \left.\frac{d}{ds}\right|_{s = 0} \zeta_n(s).
\]
\end{definition}

\begin{remark}[Analytic torsion in terms of the heat kernel]
\label{rem:analytic_torsion_in_terms-of_the_heat_kernel}
One can rewrite the analytic torsion also in terms of the heat kernel by
\begin{equation*}
\rho_{\an}(M;V)  :=  \frac{1}{2} \cdot
\sum_{n \ge 0} (-1)^n \cdot n \cdot \frac{d}{ds}
\left.
\frac{1}{\Gamma (s)} \cdot
\int_0^{\infty} t^{s-1} \cdot \theta_n(M)^{\perp} \; dt
\right|_{s = 0},
\end{equation*}
where 
\begin{eqnarray*}
\theta_n(M)(t)
& := &
\int_M \tr\bigl(e^{-t\Delta_n}(x,x)\bigr) \dvol;
\\
\theta_n(M)^{\perp} 
& = &
\theta_n(M)(t) - \dim_{\IR}(H_n(M;\IR));
\\
\Gamma(s) 
& = &
\int_0^{\infty} t^{s-1}e^{-t} dt \quad \text{for} \; \Real(s) > 0,
\end{eqnarray*}
the Gamma-function $\Gamma(s)$ is defined for $s \in \IC$ by meromorphic extension with
poles of order $1$ in $\{n \in \IZ \mid n \le 0\}$ and satisfies $\Gamma(s+1) = s \cdot
\Gamma(s)$ and $\Gamma(n+1) = n{!}$ for $n \in \IZ, n\ge 0$, see for
instance~\cite[Section~3.5.1]{Lueck(2002)}.
\end{remark}


\subsection{Analytic torsion of $S^1$ and the Riemann Zeta-function}
\label{subsec:The_analytic_torsion_of_S1}

Fix a positive real number $\mu$.  Equip $\IR$ with the standard
metric and the unit circle $S^1$ with the Riemannian metric for which
$\IR \to S^1, t \mapsto \exp(2\pi i\mu^{-1}t)$ is
isometric. Then $S^1$ has volume $\mu$.  
The Laplace operator $\Delta^1\colon   \Omega^1(\IR) \to  \Omega^1(\IR)$
sends $f(t) dt$ to $-f^{\prime\prime}(t) dt$. By checking the $\mu$-periodic solutions of
$f^{\prime\prime}(t) = -\lambda f(t)$, one shows that 
$\Delta^1 \colon \Omega^1(S^1) \to \Omega^1(S^1)$ 
 has eigenspaces 
\[
E_{\lambda}(\Delta_1) =
\begin{cases}
\operatorname{span}_{\IR}\{f_n   dt, g_n dt\} & \text{for }\; \lambda = (2\pi\mu^{-1}n)^2, n\ge 1;
\\
\operatorname{span}_{\IR}\{dt\} &  \text{for}\;  \lambda = 0;
\\
\{0\} & \text{otherwise},
\end{cases}
\]
where $f_n(exp(2\pi  i\mu^{-1}t)) = \cos(2\pi\mu^{-1}nt)$ and 
$g_n(exp(2\pi  i\mu^{-1}t)) = \sin(2\pi\mu^{-1}nt)$.
Denote by 
\begin{eqnarray}
\zeta_{\Riem}(s) = \sum_{n\ge 1}n^{-s}
\label{Riemann_Zeta_function}
\end{eqnarray}
the \emph{Riemannian Zeta-function}.
We have
\[\zeta_1(s) = \sum_{n \ge 1} 2 \cdot (2\pi\mu^{-1}n)^2)^{-s}.
\]
As $\zeta_{\Riem}(0) = -\frac{1}{2}$ and
$\zeta_{\Riem}^{\prime}(0) = -\frac{\ln(2\pi)}{2}$ hold (see
Titchmarsh~\cite{Titchmarsh(1951)}), we obtain
\begin{eqnarray*}
\rho_{\an}(S^1) 
& = & 
\frac{1}{2} \cdot \sum_{n \ge 0} (-1)^n \cdot n \cdot \left.\frac{d}{ds}\right|_{s = 0} \zeta_n(s)
\\
 & = & 
- \frac{1}{2} \cdot \left.\frac{d}{ds}\right|_{s = 0} \zeta_1(s)
\\
 & = & 
- \left.\frac{d}{ds}\right|_{s = 0} \left(\sum_{n \ge 1} ((2\pi\mu^{-1}n)^2)^{-s}\right)
\\
 & = & 
- \left.\frac{d}{ds}\right|_{s = 0} \left(\exp(-2 \cdot \ln(2\pi \mu^{-1}) \cdot s) \cdot \zeta_{\Riem}(2s)\right)
\\
 & = & 
- \left(\left.\frac{d}{ds}\right|_{s = 0} \exp(-2 \cdot \ln(2\pi \mu^{-1}) \cdot s)\right) \cdot \zeta_{\Riem}(0)
\\
& & \quad \quad  -  
\exp(-2 \cdot \ln(2\pi \mu^{-1} \cdot 0)  \cdot \left.\frac{d}{ds}\right|_{s = 0} \zeta_{\Riem}(2s)
\\
& = & 
2 \cdot \ln(2\pi \mu^{-1}) \cdot \zeta_{\Riem}(0) - 2\cdot  \left.\frac{d}{ds}\right|_{s = 0} \zeta_{\Riem}(s)
\\
& = & 
2 \cdot \ln(2\pi \mu^{-1}) \cdot \frac{-1}{2} - 2\cdot  \frac{-\ln(2\pi)}{2}
\\
& = & 
- \ln(2\pi ) +  \ln(\mu)  + \ln(2\pi)
\\
& = & \ln(\mu).
\end{eqnarray*}
Notice that this agrees with $\rho^{\topo}(S^1)$ by~\eqref{rho_topo(rational_homology_sphere)}.


\subsection{The equality of analytic and topological torsion for closed Riemannian manifolds:
The Cheeger-M\"uller Theorem}
\label{subsec:The_Cheeger-Mueller_Theorem}

The following celebrated result was proved independently 
by Cheeger~\cite{Cheeger(1979)} and M\"uller~\cite{Mueller(1978)}.

\begin{theorem}[Equality of analytic and Reidemeister torsion]
\label{the:Equality_of_analytic_and_Reidemeister_torsion}
Let $M$ be a closed Riemannian manifold. Then
\[
\rho_{\an}(M) = \rho_{\topo}(M).
\]
\end{theorem}
It was already known before the final proof of
Theorem~\ref{the:Equality_of_analytic_and_Reidemeister_torsion} that the difference
$\rho_{\an}(M) - \rho_{\topo}(M)$ is independent of the Riemannian metric and
$\rho_{\an}(M)$ and $\rho_{\topo}(M)$ satisfies analogous product formulas so that the
desired equality holds for a product $M \times N$ if it holds for both $M$ and $N$.
M\"uller's strategy was to show that the difference $\rho_{\an}(M) - \rho_{\topo}(M)$ depends
only on the bordism class of $M$ and then verify the equality on generators of the
oriented bordism ring. He also uses an interesting result of Dodziuk and
Patodi~\cite[Theorem~3.7]{Dodziuk-Patodi(1976)} that the eigenvalues of the combinatorial Laplace
operator $\Delta_n(K)$ of a smooth triangulation $K$ of $M$ converge to the eigenvalues of
the analytic Laplace operator $\Delta_n(M)$ if the mesh, which is the supremum over the
distances with respect to the metric coming from the Riemannian metric of any two vertices
spanning a $1$-simplex, of the triangulation $K$ goes to zero.


\subsection{The relation between analytic and topological torsion for compact Riemannian manifolds}
\label{subsec:The_relation_between_analytic_and_topological_torsion_for_compact_Riemannian_manifolds}

Let $M$ be a compact Riemannian manifold. Suppose that its boundary $\partial M$ is
written as disjoint union $\partial_0 M\coprod \partial_1 M$, where $\partial_i M$ itself
is a disjoint union of path components of $\partial M$.  In particular $\partial_i M$
itself is a closed manifold. We will assume that the Riemannian metric on $M$ is a product
near the boundary and we will equip $\partial M$ with the induced Riemannian metric.  By
introducing appropriate boundary condition for the Laplace operator one can define
$\rho_{\an}(M,\partial_0M)$ and $\rho_{\topo}(M,\partial_0 M)$. The next result is proved
in~\cite[Corollary~5.1]{Lueck(1993)}.

\begin{theorem}[The relation between analytic and topological torsion for compact Riemannian manifolds]
\label{the:The_relation_between_analytic_and_topological_torsion_for_compact_Riemannian_manifolds}
We get under the conditions above
\[\rho_{\an}(M,\partial_0M)  = 
\rho_{\topo}(M,\partial_0 M) + \frac{\ln(2)}{2}
\cdot \chi(\partial M).
\]
\end{theorem}

\begin{example}[Unit interval] \label{exa:I} Equip $I= [0,1]$ with the standard metric
  scaled by $\mu > 0$. The volume form is then $\mu dt$. The analytic Laplace operator
  $\Delta_1 \colon \Omega^1 I \to \Omega^1 I$ maps $f(t) dt$ to $- \mu^{-2}
  f^{\prime\prime}(t) dt$.  Denote by $E_{\lambda}(\Delta_1(I))$ and
  $E_{\lambda}(\Delta_1(I,\partial I))$ the eigenspace of $\Delta_1$ for $\lambda \ge 0$,
  where for $ \Delta_1(I)$ and $\Delta_1(I,\partial I)$ respectively we require for a
  $1$-form $f(t) dt$ the boundary condition $f(0) = f(1)= 0$ and $f'(0) = f'(1)= 0$
  respectively.  If $\lambda = (\pi \mu^{-1} n)^2$ 
  for $n \in \IZ ,n  \ge 1$, then
\begin{eqnarray*}
E_{\lambda}(\Delta_1(I)) 
& = &
\operatorname{span}_{\IR}\{\sin (\pi nt)dt\};
\\
E_{\lambda}(\Delta_1(I,\partial I)) 
& = &
\operatorname{span}_{\IR}\{\cos (\pi nt)dt)\},
\end{eqnarray*}
and $E_\lambda(\Delta_1(I))  = E_{\lambda}(\Delta_1(I,\partial I) = 0$
if $\lambda$ is not of this form. As $\zeta_{\Riem}(0) = -\frac{1}{2}$ and
$\zeta_{\Riem}^{\prime}(0) = -\frac{\ln(2\pi)}{2}$ hold (see
Titchmarsh\cite{Titchmarsh(1951)}),  we get 
\[\zeta_1(I) =
\zeta_{1}(I, \partial I) = \left( \frac{\pi}{\mu} \right)^{-2s}
\cdot \zeta_{Rie}(2s).
\]
This implies
\[
\rho_{\an}(I) = \rho_{\an}(I,\partial I) = \ln(2\mu).
\]
A calculation similar to the one of Subsection~\ref{subsec:Topological_torsion_of_rational_homology_spheres}
shows
\[
\rho_{\topo}(I) = \rho_{\topo}(I,\partial I) = \ln(\mu).
\]
This is compatible with 
Theorem~\ref{the:The_relation_between_analytic_and_topological_torsion_for_compact_Riemannian_manifolds}
since $\chi(\partial I) = 2$.
\end{example}

\begin{remark}[Twisting with finite-dimensional orthogonal representations]
  \label{rem:Twisting_with_finite-dimensional_orthogonal_representations_again}
  For an orthogonal finite-dimen\-sio\-nal
  representation $V$ of $\pi_1(M)$ for a compact  Riemannian manifold $M$
  one can also define the  $V$-twisted analytic torsion $\rho_{\an}(M, \partial_0 M;V)$ and $V$-twisted
  topological torsion $\rho_{\topo}(M,\partial_0 M;V)$. 
  Theorem~\ref{the:The_relation_between_analytic_and_topological_torsion_for_compact_Riemannian_manifolds}
  generalizes to 
\[\rho_{\an}(M,\partial_0M;V)  = 
\rho_{\topo}(M,\partial_0 ;V) + \frac{\ln(2)}{2}
\cdot \dim_{\IR}(V) \cdot \chi(\partial M).
\]
\end{remark}

\begin{remark}[Elliptic operators and indices]
\label{rem:elliptic_operators_and_indices}
The Euler characteristic term in
Theorem~\ref{the:The_relation_between_analytic_and_topological_torsion_for_compact_Riemannian_manifolds}
can be interpreted as the index of the de Rham complex. This leads to the following
question.

Let $P^*$ be an elliptic complex of partial differential operators.  Denote by
$\Delta(P^*)_*$ the associated Laplace operator. It is an elliptic positive self-adjoint
partial differential operator in each dimension. Hence its analytic torsion
$\rho_{\an}(P^*)$ can be defined as done before for the ordinary Laplace operator. Suppose
that the complex $P^*$ restricts on the boundary of $M$ to an elliptic complex $\partial
P^*$ in an appropriate sense. Can one find a more or less topological invariant
$\rho_{\topo}(P^*)$ such that the following equation holds
\[
\rho_{\an}(P^*) = \rho_{\topo}(P^*) + \frac{\ln(2)}{2} \cdot \operatorname{index}(\partial P^*).
\]
If we take $P^*$ to be the de Rham complex and put $\rho_{\topo}(P^*)$
to be the topological torsion $\rho_{\topo}(M)$, then the equation above just reduces 
to  Theorem~\ref{the:The_relation_between_analytic_and_topological_torsion_for_compact_Riemannian_manifolds}.
\end{remark}


\typeout{------------ Section 6: Equivariant torsion for actions of finite groups-------------}

\section{Equivariant torsion for actions of finite groups}
\label{sec:Equivariant_torsion_for_actions_of_finite_groups}

Throughout this section $G$ is a finite group.
Let $M$ be a compact Riemannian manifold. Suppose that its boundary $\partial M$ is
written as the  disjoint union $\partial_0 M\coprod \partial_1 M$, where $\partial_i M$ itself is
a disjoint union of path components of $\partial M$.  Let $G$ be a finite group acting by
isometries on $M$.

Let $\Rep_{\IR}(G)$ be the real representation ring of $G$.
Denote by $K_1(\IR G)^{\IZ/2}$ the $\IZ/2$-fixed point set of the $\IZ/2$-action on $K_1(\IR G)$ 
which comes from the involution of rings
$\IR G \to \IR G, \; \sum_{g \in G} r_g \cdot g \mapsto \sum_{g \in G} r_g \cdot g^{-1}$.
Then one can define \emph{the equivariant analytic torsion}
\[
\rho_{\an}^G(M,\partial_1 M) \in \IR \otimes_{\IZ} \Rep_{\IR}(G),
\]
the \emph{equivariant topological torsion}
\[
\rho_{\topo}^G(M,\partial_1 M) \in K_1(\IR G)^{\IZ/2},
\]
the \emph{Poincar\'e torsion}
\[
\rho_{\pd}^G(M,\partial_1 M) \in K_1(\IR G)^{\IZ/2},
\]
and the \emph{equivariant Euler characteristic}
\[
\chi^G(\partial M) \in \IR \otimes_{\IZ} \Rep_{\IR}(G).
\]
The analytic torsion $\rho_{\an}^G(M,\partial_0 M)$ is defined analogously
to the analytic torsion in the non-equivariant case, one just takes into account
that the eigenspaces $E_{\lambda}(\Delta_n)$ determine elements in
$\Rep_{\IR}(G)$ and counts it as an element in $\Rep_{\IR}(G)$ instead of only
counting its dimension. The topological torsion is defined in terms of the cellular chain complex
of an equivariant triangulation. The equivariant Euler characteristic $\chi^G(\partial_0 M)$
is given by $\sum_{n \ge 0} (-1)^n \cdot [H_n(\partial M;\IR)]$ taking again into account
that $H_n(\partial M;\IR)$ is a finite-dimensional $G$-representation. 
A new phenomenon is represented by the Poincar\'e torsion $\rho_{\pd}^G(M,\partial_0 M)$
which measures the deviation from equivariant Poincar\'e duality being simple. It is defined
in terms of the $\IZ G$-chain map
$- \cap [M] \colon C^{\dim(M) - *}(M,\partial_1M) \to C_*(M,\partial_0 M)$
which is a $\IZ$-chain homotopy equivalence but not necessarily a 
$\IZ G$-chain homotopy equivalence.
If $M$ has no boundary and has odd dimension,  or  if $G$ acts freely, then
$\rho_{\pd}^G(M,\partial_0 M)$ vanishes.

Denote by $\widehat{\Rep}_{\IR}(G)$ the subgroup of $\Rep_{\IR}(G)$
generated by the irreducible representations of real or complex type.
The following result is proved in~\cite[Theorem~4.5]{Lueck(1993)}.

\begin{theorem}[Equivariant torsion]
\label{the:Equivariant_torsion}
Suppose that the Riemannian metric on $M$ is a product near the boundary.
Then there is an isomorphism
\[
\Gamma_1 \oplus \Gamma_2: K_1(\IR G)^{\IZ/2} \xrightarrow{\cong}
\IR \otimes_{\IZ} \Rep_{\IR}(G) \oplus (\IZ/2 \otimes_{\IZ} \widehat{\Rep}_{\IR}(G),
\]
and we have
\[\rho_{\an}^G(M,\partial_1 M)   =  \Gamma_1\bigl(\rho_{\topo}^G(M,\partial_1 M)\bigr) 
- \frac{1}{2}\cdot \Gamma_1\bigl(\rho_{\pd}^G(M,\partial_1 M)\bigr) + \frac{\ln(2)}{2}\cdot \chi^G(\partial M),
\]
and
\[
\Gamma_2\bigl(\rho_{\topo}^G(M,\partial_1 M)\bigr)  = \Gamma_2\bigl(\rho_{\pd}^G(M,\partial_1 M)\bigr)  =  0.
\]
\end{theorem}

\begin{remark}[The strategy of proof]
  \label{rem:the_strategy_of_proof}
  Let $M$ be a compact $G$-manifold with $G$ invariant Riemannian metric which is a
  product near the boundary.  Then its double $M \cup_{\partial M} M$ inherits a 
 $G   \times \IZ/2$-action and a $G \times \IZ/2$-invariant metric.  It turns out that the
  equivariant torsion $\rho^{G \times \IZ/2}_{\an}(M \cup_{\partial M} M)$ carries the same 
  information as $\rho^{G}_{\an}(M)$ and    $\rho^{G}_{\an}(M,\partial M)$ together. This is also true for 
  $\rho^{G \times \IZ/2}_{\topo}(M \cup_{\partial M} M)$ but the concrete formulas are different for the topological and
  analytical setting, the difference term is essentially $\frac{\ln(2)}{2}\cdot   \chi^G(\partial M)$.  
  Thus one can reduce the case of  a compact $G$-manifold to  the a case
  of a closed $G \times \IZ/2$-manifold. 

  Lott-Rothenberg~\cite{Lott-Rothenberg(1991)} handled  the odd-dimensional case without boundary 
  using ideas of Cheeger~\cite{Cheeger(1979)} and M\"uller~\cite{Mueller(1978)}.
  They noticed that in the even-dimensional case the analytic and topological torsion do not agree
without computing the correction term, which turns out to be the Poincar\'e torsion.
\end{remark}

\begin{remark}[Unit spheres in representations]
\label{rem:Unit_spheres_in_representations}
The Poincar'e duality torsion can be used to reprove  the 
celebrated result of de Rham~\cite{deRham(1964)} that two orthogonal 
$G$-representations $V$ and $W$ are isometrically $\IR G$-isomorphic 
if and only if their unit spheres are 
$G$-diffeomorphic, see~\cite[Section~5]{Lueck(1993)}.
Similar proofs can be found in Rothenberg~\cite{Rothenberg(1978)} and 
Lott-Rothenberg~\cite{Lott-Rothenberg(1991)}. The result is an extension 
of the classification of lens spaces which is carried out for example in 
Cohen~\cite{Cohen(1973)} and Milnor~\cite{Milnor(1966)}. \par

The result of de Rham does not hold in the topological 
category. Namely, there are non-linearly isomorphic $G$-representations 
$V$ and $W$ whose unit spheres are $G$-homeomorphic, see 
Cappell-Shaneson~\cite{Cappell-Shaneson(1981)}), and
also~\cite{Cappell-Shaneson-Steinberger-Weinberger-West(1990),Cappell-Shaneson-Steinberger-West(1989), 
Hambleton-Pedersen(2007)}.

However, if $G$ has odd  order, $G$-homeomorphic implies $G$-diffeo\-mor\-phic for unit spheres in 
$G$-representations as shown by Hsiang-Pardon~\cite{Hsiang-Pardon(1982)} 
and Madsen-Rothen\-berg~\cite{Madsen-Rothenberg(1988a)}.
\end{remark}

\begin{example}[$S^1$ with complex conjugation] \label{exa:S1_with_complex-conjugation}
Fix a positive real number $\mu$.  Equip $\IR$ with the standard
metric and the unit circle $S^1$ with the Riemannian metric for which
$\IR \to S^1, t \mapsto \exp(2\pi i\mu^{-1}t)$ is
isometric. Then $S^1$ has volume $\mu$.  Let $\IZ/2$ act on $S^1$ by complex
conjugation. We get by a direct  computation, see~\cite[Example~1.15]{Lueck(1993)}
\[
\rho_{\an}^{\IZ/2}(S^1) = \ln(\mu ) \cdot ([\IR] + [\IR^-]) \in \IR \otimes_{\IZ} \Rep_{\IR}(\IZ/2),
\]
where $\IR$ is the trivial $1$-dimensional real $\IZ/2$-representation and $\IR^-$ is the
$1$-dimensional $\IZ/2$-representation for which the generator of $\IZ/2$ acts by
$-\id_{\IR}$.  We obtain in $K_1(\IR[\IZ/2])^{\IZ/2}$ by a direct computation,
see~\cite[Example~3.25]{Lueck(1993)},
\begin{eqnarray*}
\rho_{\topo}^{\IZ/2}(S^1) 
& = &
[\mu/2 \cdot \id \colon \IR \to \IR] + [2\mu \cdot \id \colon \IR^- \to \IR^-];
\\
\rho_{\pd}^{\IZ/2}(S^1) 
& = &
[4 \cdot \id \colon \IR \to \IR] + [1/4 \cdot \id \colon \IR^- \to \IR^-].
\end{eqnarray*}
This is compatible with Theorem~\ref{the:Equivariant_torsion}.
\end{example}

\begin{example}[$S^1$ with antipodal action] \label{exa:S1_with_antipodal_action}
Fix a positive real number $\mu$.  Equip $\IR$ with the standard
Riemannian metric and the unit circle $S^1$ with the Riemannian metric for which
$\IR \to S^1, t \mapsto \exp(2\pi i\mu^{-1}t)$ is
isometric. Then $S^1$ has volume $\mu$.  Let $\IZ/2$ act on $S^1$ by the antipodal map
which sends $z$ to $-z$.
This is a free orientation preserving action.
Then 
\begin{eqnarray*}
\rho_{\an}^{\IZ/2}(S^1) 
& = & 
\ln(\mu ) \cdot [\IR[\IZ/2]];
\\
\rho_{\topo}^{\IZ/2}(S^1) 
& = &
[\mu \cdot \id \colon \IR[\IZ/2] \to \IR[\IZ/2]];
\\
\rho_{\pd}^{\IZ/2}(S^1) 
& = &
0.
\end{eqnarray*}
\end{example}

As an illustration we state the following corollary of Theorem~\ref{the:Equivariant_torsion}
and basic considerations about Poincar\'e duality, which explains the role of
the Poincar\'e torsion that  does not appear in the non-equivariant setting,
see~\cite[Corollary~5.6]{Lueck(1993)}.

\begin{corollary} \label{cor:orientable_case} Let $M$ be a Riemannian $G$-manifold
with invariant  Riemannian metric.  Suppose that $M$ is 
closed and orientable and $G$ acts orientation preserving. 

\begin{enumerate} 
\item \label{cor:orientable_case:odd_dimension} 
If $dim(M)$ is odd , we have
\begin{eqnarray*}
\rho_{\an}^G(M) 
& = & 
\Gamma_1(\rho_{\topo}(M));
\\
\rho_{\pd}^G(M) & = & 0;
\end{eqnarray*}
\item \label{cor:orientable_case:even_dimension} 
If $dim(M)$ is even,  we get
\begin{eqnarray*}
\rho_{\an}(M) & = & 0;
\\
\rho_{\topo}^G(M) & = & \frac{\rho_{\pd}^G(M)}{2}.
\end{eqnarray*}
\end{enumerate}
\end{corollary}

\begin{remark}[Twisting with equivariant coefficient systems]
  \label{rem:Twisting_with_coefficient}
  There are also versions of the notions and results  of this section for appropriate 
equivariant coefficient system as explained in~\cite{Lueck(1993)}.
\end{remark}


\typeout{-------------------------------------- Section 7: Outlook------------------------------}

\section{Outlook}
\label{sec:Outlook}


\subsection{Analytic torsion}
\label{subsec:Analytic_torsion}

There are many important papers about analytic  torsion and variations of it
in the literature.  We have to leave it to the reader to figure out the relevant
authors and papers since an appropriate discussion would go far beyond the scope of this
article. At least we give a list of references which is far from being complete.
They concern for instance determinant lines, holomorphic versions, higher versions, equivariant versions,
singular spaces, algebraic varieties, and hyperbolic manifolds, 
see~\cite{Bismut(1993),Bismut(1997fam),Bismut(1998), Bismut-Gillet-Soule(1988a), Bismut-Gillet-Soule(1988b),
Bismut-Gillet-Soule(1988c), Bismut-Goette(2000a),Bismut-Goette(2000b),Bismut-Goette(2001),
Bismut-Koehler(1992), Bismut-Lott(1993), Bismut-Lott(1995), Bismut-Zhang(1992),  Bismut-Zhang(1994),
Braverman-Kappeler(2008),Bunke(2000),
Burghelea(1997), Burghelea-Friedlander-Kappeler(1998),
Burghelea-Haller(2001), Burghelea-Haller(2006Turaev), Burghelea-Haller(2006Riemannian), 
Burghelea-Haller(2007), Burghelea-Haller(2008dynamics), Burghelea-Haller(2008function), 
Burghelea-Haller(2010), Burgos-Freixas-Litcanu-Ruazvan(2014),Cappell-Miller(2010), 
Dai-Melrose(2012), Farber(1995), Farber(1996a),Felshtyn(2000), Fried(1986a), Koehler-Weingart(2003),
Lesch(2013),Lott(1994Lie), Lott(1999),Mazzeo-Vertman(2012),Mourougane(2006),Mueller(1993),Mueller(2012), 
Mueller-Pfaff(2013locsym), Ray-Singer(1973),Vishik(1995),Yoshikawa(2004),Yoshikawa(2006),Yoshikawa(2013)}.


\subsection{Topological torsion}
\label{subsec:topological_torsion}

Also for topological torsion there are many important papers  in the literature. In particular
Whitehead torsion and Reidemeister torsion have been intensively studied.
Again we have to leave it to the reader to figure out the relevant
authors and papers since an appropriate discussion would go far beyond the scope of this
article. At least we give a list of references which is  far from being complete,
they concern for instance $s$-cobordisms, knot theory, classification of manifolds, equivariant versions,
and higher versions,
see~\cite{Badzioch-Dorabia-Klein-Williams(2011),Bartels-Lueck(2012annals), 
Chapman(1973), Chapman(1974),Cohen(1973),Igusa(2002),Igusa(2008),Kervaire(1965), Kirby-Siebenmann(1977), 
Lueck(1989), Lueck(2002c), Lustig-Moriah(1993b),Madsen(1983), Milnor(1961), Milnor(1962a), Milnor(1966), Nicolaescu(2003),
Oliver(1989), Ranicki(1985_torsI), Ranicki(1987_torsII), Ranicki(1996b), 
Reidemeister(1938), Siebenmann(1970), Steimle(2012), Turaev(1986),
Whitehead(1941_incnuclei),Whitehead(1949_comHomI,Whitehead(1950)}.


\subsection{$L^2$-versions}
\label{subsec:L2-versions}

The next level is to pass to non-compact spaces, essentially to a $G$-covering 
$\overline{M} \to M$ for a closed Riemannian manifold $M$ and the induced $G$-invariant 
Riemannian metric on $\overline{M}$ or to a $G$-covering $\overline{X} \to X$ for a finite
$CW$-complex $X$, where $G$ is a (not necessarily finite) discrete group.
 This requires to extend our basic invariants of finite-dimensional
Hilbert spaces of Sections~\ref{sec:Operators_of_finite-dimensional_Hilbert_spaces}
and~\ref{sec:Finite_Hilbert_chain_complexes} to an appropriate setting of infinite
dimensional Hilbert spaces taking the cocompact free proper group action on $\overline{M}$ or
$\overline{X}$ into account.  Here group von Neumann algebras play a key role. The
first instance where this has been carried out is the paper by Atiyah~\cite{Atiyah(1976)}.
Generalizations of  the ideas about the spectral density function presented in
Subsection~\ref{subsec:The_spectrum_and_in_the_spectral_density_function}
and~\ref{subsec:Rewriting_determinants}.
come into play and lead for instance to the notion of the Fuglede-Kadison determinant,
which generalizes the classical determinant to this setting. The material presented in 
Subsection~\ref{subsec:The_spectrum_and_in_the_spectral_density_function}
and~\ref{subsec:Rewriting_determinants} is  helpful 
if one wants to  understand the $L^2$-versions.

All this leads to the notions of $L^2$-Betti numbers and of $L^2$-torsion, which have been
intensively studied in the literature and have many applications to problems arising in
topology, geometry, group theory and von Neumann algebras. A discussion of these $L^2$-invariants
would go far beyond the scope of this
article. For more information about
the circle of these ideas and invariants we refer for instance
to~\cite{Lueck(2002),Lueck(2009algview),Lueck(2015_l2approx)}.

\typeout{-------------------------------------- References  ---------------------------------------}


\end{document}